\documentclass[options]{amsart}
  \pdfoutput=1
\usepackage{hyperref}
\usepackage{cleveref}
\usepackage{latexsym,ifthen,amssymb}
\usepackage{latexsym,ifthen,amssymb}
\usepackage[toc,page,title,titletoc,header]{appendix}
\usepackage{multirow}
 \usepackage{longtable}
 \usepackage{rotating}
 \usepackage{graphicx, amsmath, amsthm, amssymb,float,setspace,color, multirow}
\usepackage{enumerate}
\usepackage{latexsym,ifthen,amssymb}
\usepackage[toc,page,title,titletoc,header]{appendix}
\usepackage{enumerate}
\usepackage{graphicx}
\usepackage{bm}
\usepackage{subfigure}
\usepackage{float}
\usepackage{rotating}
\usepackage{epstopdf}
\usepackage{hyperref}

\usepackage{url}
\usepackage[margin=1in]{geometry}
\usepackage{xcolor}	
\usepackage{soul}

\usepackage{thmtools}
\usepackage{thm-restate}

\newtheorem{theorem}{Theorem}[section]
\newtheorem{lemma}[theorem]{Lemma}

\newtheorem{claim}[theorem]{Claim}
\newtheorem{corollary}[theorem]{Corollary}
\newtheorem{proposition}[theorem]{Proposition}

\theoremstyle{definition}
\newtheorem{definition}[theorem]{Definition}

\theoremstyle{remark}
\newtheorem{remark}[theorem]{Remark}
\newtheorem {question}[theorem]{Question}
\numberwithin{equation}{section}

\newtheoremstyle{noparens}%
  {}{}%
{}{}%
{\bfseries}{.}%
{ }%
{\thmname{#1}\thmnumber{ #2}\thmnote{ #3}}
\theoremstyle{noparens}
\newtheorem*{question*}{Question}
\newtheorem*{theorem*}{Theorem}

\newcommand{\uhr}{{\upharpoonright}}

\newcommand{\leaf}{\operatorname{\ell}}

\def\t{\tilde}

\def\msf{\mathsf}
\def\mcal{\mathcal}

\def\<{<^*}
\def\>{>^*}

\def\hyperimmune{hyperimmune}

\def\wstable{weakly stable}

\def\wtcoh{\mathsf{wCTT}_k^2}
\def\tcoh{\mathsf{CTT}_k^2}

\def\PA{PA}

\def\treeset{tree}
\def\somewheredense{somewhere dense}

\def\tt{\msf{TT}}
\def\stt{\msf{STT}}

\def\treesplit{tree-split}

\def\conciseversion{refinement}
\def\concised{refined}
\def\cross{cross product}

\def\hierarchy{hierarchy}

\def\twobranching{$2$-branching}

\def\partition{partition}

\def\disperse{scattered}

\def\h{\hat}
\def\v{\vec}

\def\k{k}

\def\mbq{\v{Q}}

\def\sufficient{sufficient}

\title{The Strength  of  Ramsey's  Theorem For Pairs over trees:\\ I.  Weak K\"onig's Lemma}

 \author{C.~T.~Chong}
\address{
Department of Mathematics, National University of Singapore, Singapore 119076.}
\email{chongct@nus.edu.sg}

\author{Wei Li}
\address{
 Department of Mathematics, National University of Singapore, Singapore 119076.}
\email{matliw@nus.edu.sg}

\author{Lu Liu}
\address{Department of Mathematics,
Central South University,
City Changsha, Hunan Province,
China. 410083}
\email{g.jiayi.liu@gmail.com}

\author{Yue Yang }
 \address{
 Department of Mathematics, National University of Singapore, Singapore 119076.}
\email{matyangy@nus.edu.sg}

\thanks{ Chong's research was partially supported by NUS grants C-146-000-042-001 and WBS : R389-000-040-101.
Liu's research was partially supported by NUS grant R389-000-040-101
and Natural Science Foundation of Hunan Province of China
2018JJ3623. Yang's research was partially supported by NUS AcRF Tier 1 grant R146-000-231-114 and MOE2016-T2-1-019.
}


\begin{document}
 \maketitle

 \begin{abstract}
 Let $\tt^2_k$ denote the combinatorial principle stating  that every $k$-coloring of pairs of compatible nodes in the full binary tree has a homogeneous solution, i.e.~an isomorphic subtree in which all pairs of compatible nodes have the same color. Let $\mathsf{WKL}_0$ be the subsystem of second order arithmetic  consisting of the base system $\mathsf{RCA}_0$ together with the principle (called Weak K\"onig's Lemma)  stating that every infinite subtree of the full binary tree has an infinite path. We show that over $\mathsf{RCA}_0$,
   $\tt^2_k$ doe not imply $\mathsf{WKL}_0$. This solves the open problem on the relative strength between the two major subsystems of second order arithmetic.

 \end{abstract}

\section{Introduction}
\label{secintro}
Let $k, n\in\mathbb{N}$ and let $[\mathbb{N}]^n$ denote the collection of $n$-element subsets of the set of natural numbers $\mathbb{N}$.
Ramsey's theorem for $[\mathbb{N}]^n$ in $k$ colors ($\mathsf{RT}^n_k$) states that every such coloring has a homogeneous set,
i.e.~an infinite set all of whose $n$-element subsets have the same color (we only consider $k\ge 2$ as $k=1$ is immediate).
The proof-theoretic strength of $\mathsf{RT}^n_k$ is a subject of major interest  in reverse mathematics in recent  years,
inspired by the seminal works of Seetapun and Slaman \cite{Seetapun1995strength}), and Cholak, Jockusch and Slaman \cite{Cholak2001strength}.
Recall that $\mathsf{RCA}_0$ denotes the base system in reverse mathematics for second order arithmetic.
The overall picture that emerges from these investigations is that over $\mathsf{RCA}_0$,
$\mathsf{RT}^n_k$ is equivalent to the arithmetical comprehension axiom system $\mathsf{ACA}_0$ when $n\ge 3$
(a corollary of Jockusch \cite{Jockusch1972Ramseys}), and is strictly weaker than $\mathsf{ACA}_0$ when $n=2$ (\cite{Seetapun1995strength}).
Indeed the case $n=2$ is of particular interest and constitutes the bulk of the effort and energy invested in this subject.
It is now known that (i) $\mathsf{RT}^2_k$ does not imply (hence not comparable with) $\mathsf{WKL}_0$,
the subsystem of second order arithmetic  which adds to $\mathsf{RCA}_0$ the principle $\mathsf{WKL}$ (Weak K\"onig's Lemma) stating that every infinite binary tree has an infinite path
(Liu \cite{Liu2012RT22} and \cite{Liu2015Cone} (for a strengthening of the result to the Weak Weak K\"onig's Lemma principle $\msf{WWKL}$ (see below))),
(ii) $\mathsf{RT}^2_k$ is strictly stronger than its stable counterpart $\mathsf{SRT}^2_k$,
and does not imply the induction scheme for $\Sigma^0_2$-formulas (Chong, Slaman and Yang \cite{Chong2014metamathematics} and \cite{chong2017inductive} respectively),
and (iii) $\mathsf{RT}^2_k$ does not prove new $\Pi^0_3$-statements about arithmetic over $\mathsf{RCA}_0$ (Patey and Yokoyama \cite{Patey2015proof}.

$\mathsf{RT}^n_k$ has a natural generalization to the Cantor space.  Let $2^{<\omega}$ denote the full binary tree (we use $\mathbb{N}$ and $\omega$ interchangeably to denote the set of natural numbers).
One now considers $k$ colorings of sets of size $n$ of compatible nodes in $2^{<\omega}$ (denoted as $[2^{<\omega}]^n$).
The principle $\tt^n_k$ states that every coloring of $[2^{<\omega}]^n$ in $k$ colors has a homogeneous solution,
i.e.~a subtree isomorphic to $2^{<\omega}$ in which all size $n$ compatible nodes have the same color.
This generalization raises a new set of interesting questions, not least because a homogeneous tree has to be topologically the full binary tree.
The isomorphism requirement is implicit in $\mathsf{RT}^n_k$ and automatically satisfied by $\mathbb N$
since it is an intrinsic property of $\mathbb{N}$ that all infinite subsets are order isomorphic.
This advantage no longer holds for infinite subtrees of $2^{<\omega}$  and new technical challenges have to be overcome to produce a solution of an instance of $\tt^n_k$ with a prescribed property.

By identifying a node in $2^{<\omega}$ with its length, it is immediate that every instance of $\mathsf{RT}^n_k$ induces an instance of $\tt^n_k$.
This implies that $\mathsf{RT}^n_k$ is a consequence of $\tt^n_k$ over $\mathsf{RCA}_0$.
In particular, one concludes from this that $\tt^n_k$, just like $\mathsf{RT}^n_k$, is equivalent to $\mathsf{ACA}_0$ for $n\ge 3$.
Dzhafarov and Patey \cite{DzhafarovColoring} have proved that $\tt^2_k$ is strictly weaker than $\tt^n_k$ for $n\ge 3$,
and it is not difficult to see that for each $k\in\mathbb{N}$, $\mathsf{RT}^1_k$ and $\tt^1_k$ are consequences of $\mathsf{RCA}_0$
(the situation is dramatically different for $\tt^1$ which is syntactically defined in the language of second order arithmetic as
$(\forall k) \tt^1_k$, see Corduan, Groszek and Mileti \cite{Corduan2010Reverse} and Chong, Li, Wang and Yang \cite{chong2017inductive}).
On the other hand, Patey \cite{Patey2015strength} has shown  that $\mathsf{RT}^2_k$ does not imply $\tt^2_k$,  presenting the first example of a Ramsey type theory
whose proof-theoretic strength lies strictly between $\mathsf{RCA}_0+\mathsf{RT}^2_k$ and $\mathsf{ACA}_0$,
Thus the central problem  concerning the status of tree colorings again revolves around the case $n=2$. In particular,
 where does $\tt^2_k$ stand {\it vis-\`a-vis}  the ``big five'' systems in reverse mathematics, the first three of which, in increasing strength, are $\mathsf{RCA}_0$, $\mathsf{WKL}_0$ and $\mathsf{ACA}_0$?  Since $\tt^2_2+\mathsf{RCA}_0\not\rightarrow \mathsf{ACA}_0$ by Dzhafarcv and Patey \cite{DzhafarovColoring}, the question is then whether
 $\tt^2_k$ implies  $\mathsf{WKL}_0$ over  $\msf{RCA}_0$ (see  \cite{DzhafarovColoring}).
 The main result of this paper answers this question:

\vskip.1in

\begin{restatable}{theorem}{main} {\rm (Main Theorem)}
\label{coro102}
Over $\mathsf{RCA}_0$, $\tt^2_k$ does not imply $\mathsf{WWKL}$. Hence $\mathsf{RCA}_0+\tt^2_k\not\rightarrow \mathsf{WKL}_0$.
\end{restatable}

\vskip.1in

$\msf{WWKL}$ is the principle introduced by Yu and Simpson \cite{Yu1990Measure} stating that every infinite binary tree $T$ of positive measure has an infinite path,
where $T$ has positive measure if there is a positive rational number $r\le 1$ such that at every level $s$ of $T$, there are at least $2^sr$ many nodes.
Thus Theorem 1.1 says  that  the extra computational power and structural complexity vested in a tree are not sufficient to prove weak K\"onig's lemma.

One can decompose $\tt^2_k$ into the sum of two combinatorial principles, the cohesive tree principle $\msf{CTT}^2_k$ and the stable tree principle $\msf{STT}^2_k$
(Dzhafarov,  Hirst and Lakins \cite{Dzhafarov2009polarized}; see Definitions \ref{defstt} and \ref{deftcoh}).
Our proof proceeds by first showing  in  Section \ref{tt2wklsec3}
 that over $\mathsf{RCA}_0$, $\msf{CTT}^2_k$  does not imply $\msf{WWKL}$ (Corollary  \ref{coro103}). To do this, we introduce a new principle called the $k$-tree-split principle ($k$-$\mathsf{TSP}$) and show that $\mathsf{CTT}^2_k$ is the sum of the weak $\mathsf{CTT}^2_k$ principle ($\msf{wCTT}^2_k$)  and $k$-$\mathsf{TSP}$, and  each of them does not  imply $\mathsf{WWKL}$.
  Then in  Section   \ref{tt22vswklsec4}  we establish  the corresponding result for $\stt^2_2$
 (Theorem  \ref{th101}).
These are achieved by showing that $\msf{wCTT}^2_k$, $k$-$\msf{TSP}$  and $\msf{STT}^2_k$ each   satisfy a property called
{\it  avoidance of bounded enumeration} (Definition \ref{defavoidance}).
Theorem \ref{coro102} follows as a consequence.
In the next section, we fix the notations and terminologies  to be used in the paper and formally define the combinatorial principles to be considered.
The final section presents a list of questions.

\section{Preliminaries}

\subsection{Notations and terminologies}\label{tt22wklpre}

We use $i,k,n,m, d$ to denote natural  numbers and
 identify a number $k\in\omega$ with the set $\{0,\dots,k-1\}$ and write $\mcal{P}(k)$ for $\{A: A\subseteq k\}$.

Denote strings (or, equivalently, nodes) in $2^{<\omega}$ by Greek letters $\rho, \sigma, \tau,\dots$.
We say $\sigma$ is extended by $\tau$ (written $\sigma\preceq\tau$ or $\tau\succeq \sigma$) if it is an initial segment of $\tau$.
The symbol $\prec$ is reserved for proper initial segment, including that of an infinite set $X\subseteq \omega$
(upon identifying $X$ with its characteristic function).
A pair of strings $\rho_0, \rho_1$ are incompatible, written as $\rho\nmid \rho_1$, if neither is extended by the other.
A set $B\subseteq 2^{<\omega}$ is \emph{prefix free}, also called  an {\it antichain}, if any two $\rho, \sigma\in  B$ are incompatible.
For emphasis, such a $B$ is sometimes   written as  $\vec{\sigma} = (\sigma_0,\dots,\sigma_n)$, where $\sigma_i \nmid \sigma_j$ when $i\ne j$.
We write $\vec{\rho} = (\rho_0,\dots,\rho_n)\succeq \vec{\sigma}= (\sigma_0,\dots,\sigma_n)$
if $\rho_i\succeq \sigma_i$ for all $i\leq n$.
We also abuse the notation $\vec{\sigma}$ and regard $\vec{\sigma}= (\sigma_0,\dots,\sigma_n)$ as the finite set $\{\sigma_0,\dots,\sigma_n\}$.
For example, we interpret $\vec{\rho}\subseteq X$ and $\vec{\rho}\cap \vec{\tau}$ in the set-theoretic sense.
For a string $\rho \in 2^{<\omega}$, we let $[\rho]^{\preceq}=\{\sigma: \sigma\succeq \rho\}$;
similarly, for $X\subseteq 2^{<\omega}$, let $[X]^{\preceq} = \{\sigma: \sigma\succeq \rho\text{ for some }\rho\in X\}$.  Let $\mathsf{Fin}(X)$ be the collection of all finite subsets of $X$.

A \emph{tree} $T\subseteq 2^{<\omega}$ is a set of strings endowed with a structure determined by the binary relation $\preceq$ (we do not require $T$  to be closed under initial segments).
Members of $T$ are also referred to as nodes of $T$.
We write $|\rho|_T$ for the $T$-length of $\rho$, i.e.~$|\rho|_T =n+1$ where $n$ is the number of proper initial segments of $\rho$ in $T$.
If $T=2^{<\omega}$, we simply write $|\rho|$.

For a set $F\subseteq 2^{<\omega}$, we write $F\uhr_y $ for $\{\rho\in F: |\rho|_F\leq y\}$;
and write $F>y$ if $|\rho|>y$ for all $\rho\in F$.
We use $\leaf(T)$ to denote the set of leaves of $T$ when $T\ne\emptyset$
(i.e.~the set of nodes in $T$ with no proper extension in $T$).
Define $\leaf(\emptyset) = \{\varepsilon\}$ where $\varepsilon$ denotes the  empty string.
A \treeset\ $T$ is \emph{$l$-branching} over $\rho$ if $T\cap [\rho]^\preceq\neq \emptyset$
and for every $\sigma\in T\cap [\rho]^\preceq$,
there exist at least $l$ pairwise  incompatible immediate extensions $\rho_0,\rho_1,\dots, \rho_{l-1}$ of $\sigma$ in $T$.
$T$ is $l$-branching if $T$ is $l$-branching over $\varepsilon$.
Given \treeset s $F,F'\subseteq 2^{<\omega}$, written $F'\succeq F$ ($F' $ {\em extends} $F$),
if $F\subseteq F'$ and $F'\setminus F\subseteq [\leaf(F)]^\preceq$.
A (finite) \emph{perfect} tree is a tree that is isomorphic to $2^{<n}$ for some $n$.
An infinite perfect tree is one that is isomorphic to $2^{<\omega}$.
Further notations will be introduced at places where they are immediately used.


We assume that the reader is familiar with the basic notions of reverse mathematics as presented in Simpson \cite{Simpson2009Subsystems}.
A (standard) model of $\mathsf{RCA}_0$ is denoted $\mathfrak M=(\omega, \mcal{S})$,
where $\mathcal S$ is a subset of the power set of $\omega$  closed under recursive join and Turing reduction.
Combinatorial objects such as trees can be coded as subsets of natural numbers, thus can be viewed as members of $\mathcal S$.


\subsection{The  main result}
We begin by recalling the combinatorial principles about trees and tree colorings that underlie  the subject matter of this paper.
First, a principle weaker than $\mathsf{WKL}$ is the following introduced by Yu and Simpson \cite{Yu1990Measure}:
(following Yu and Simpson, we use $\mathsf{WKL}$ for the combinatorial principle, and use $\mathsf{WKL}_0$ for the subsystem
$\mathsf{RCA}_0+\mathsf{WKL}$ of second order arithmetic.)
\begin{definition}
The principle {\em weak weak K\"onig's lemma} ($\mathsf{WWKL}$) states that for every infinite tree $T\subseteq 2^{<\omega}$,
if there is a rational number $r>0$ such that $|T_s|/2^s>r$ for all $s\in\omega$, where $T_s=\{\sigma : \sigma\in  T\wedge |\sigma|_T = s\}$  and
$|T_s|$ is the cardinality of  $T_s$,  then there is an infinite path in $T$.
\end{definition}
It is known that , $\mathsf{WKL}_0$ is strictly stronger than $\mathsf{RCA}_0+\mathsf{WWKL}$ \cite{Yu1990Measure}.
We produce a model of $\mathsf{RCA}_0+\tt^2_k$ in which $\mathsf{WWKL}$ (hence $\msf{WKL}$) fails. i.e.:

\main*

The model $\mathfrak M$ we construct for Theorem \ref{coro102} will be obtained from solutions of instances of $\tt^2_k$ that satisfy
the property of avoidance of bounded enumeration, defined in Definition \ref{defavoidance}.
To set the stage for the proof, we recall a number of combinatorial principles which are related or refinements of $\tt^2_k$.
As for $\mathsf{RT}^2_k$, one can introduce a notion of stability for coloring of pairs of compatible nodes of a tree.
However, the intrinsic topological structure of a tree entails that there are several possibilities for generalizing  this notion that differ in proof-theoretic strength:



\begin{definition}
A $k$-coloring $C:[2^{<\omega}]^2\rightarrow \k$ is \emph{\wstable} on an infinite tree $T$  if for every $\sigma\in T$,
there exists an $n$ such that for every $\rho\succeq \sigma$ in $T$ with $|\rho|_T= n$,
there exists a $\k'\in \k$ and $C\big(\{\sigma,\tau\}\big)=\k'$ for every $\tau\in T$ such that $\tau\succ\rho$.

A $k$-coloring $C:[2^{<\omega}]^2\rightarrow \k$ is \emph{stable} if for every $\sigma$,
there exists a $\k'\in \k$ such that for all but finitely many $\rho\succeq \sigma$,
$C\big(\{\sigma,\rho\}\big)=\k'$. Clearly stability implies weak stability in a tree.
\end{definition}

\begin{definition}\label{defstt}
The  stable tree theorem for pairs principle ($\stt^2_k$) states that
every stable $k$-coloring $C$ of the full binary tree admits an infinite perfect subtree  $T$ such that $|C\restriction [T]^2|=1$.
\end{definition}

\begin{definition} \label{deftcoh}
The  principle of weakly cohesive (resp.~cohesive) tree theorem for pairs $\wtcoh$ (resp.~$ \tcoh$) states that
every $k$-coloring $C$ of the full binary tree admits an infinite perfect  subtree $T$ such that $C\restriction [T]^2$
is \wstable\ (resp.~ stable).
\end{definition}

As in \cite{Cholak2001strength} where $\mathsf{RT}^2_k$ was decomposed into the sum of the cohesive principle $\mathsf{COH}$ and the stable Ramsey's theorem  principle $\mathsf{SRT}^2_k$,
one has a corresponding decomposition of $\tt^2_k$:
\begin{proposition} [Dzhafarov, Hirst and Lankins \cite{Dzhafarov2009polarized}] \label{DHLdecomp}
Over $\msf{RCA}_0$, $\msf{TT}_k^2 \leftrightarrow \msf{STT}_k^2+ \tcoh$.
\end{proposition}

Each of the principles defined above may be expressed as a $\Pi^1_2$-sentence $\theta$ of the form $\forall X\exists Y\varphi(X,Y)$ where $\varphi$ is arithmetical.
We will call $\theta$ a {\em problem} $\mathsf{P}$, $X$ an {\em instance} of $\mathsf{P}$ and $Y$ a {\em $\msf{P}$-solution} of $X$.
In the case of $\tt^2_k$, the combinatorial principle itself is a problem, each coloring $C: [2^{<\omega}]^2\rightarrow k$ is an instance of the problem,
and an infinite perfect \treeset\ $T$ with $|C([Y]^2)|=1$ is a solution of  the problem for the instance $C$.

\begin{remark}
It is worth noting, however, that while $\stt^2_k$ may be considered a natural generalization of $\mathsf{SRT}^2_k$,
there is little resemblance between $\mathsf{COH}$, which is defined based on the notion of an array, and $\mathsf{CTT}^2_k$.
In fact, there is a difference in terms of extendability between the two:
every finite set of numbers extends to a solution of a given instance of $\msf{COH}$ while this is not true for $\tcoh$ for $k\ge 2$,
in that not every finite perfect tree is extendible to an infinite perfect tree that solves a given instance of $\tcoh$.
In this respect, $\wtcoh$ resembles $\msf{COH}$ the most. In \cite{DzhafarovColoring} a bushy tree forcing method was employed to prove the cone avoidance property for $\tcoh$.
While this leads to the proof that $\tt^2_k\not\rightarrow \mathsf{ACA}_0$ over $\mathsf{RCA}_0$,
it does not appear to be sufficient for proving $\mathsf{RCA}_0+\tcoh\not\rightarrow\msf{WKL}_0$ since there exists a $\Pi_1^0$-class $Q$ of bushy
trees (of appropriate width) in which every infinite perfect subtree of a $T\in Q$ is of \PA-degree. Finally, it can be shown that for $k\ge 2$, $\mathsf{RCA}_0+\mathsf{CTT}^2_k$ implies the $\Sigma^0_2$-bounding induction scheme, while it is not the case for $\mathsf{COH}$.
\end{remark}


\section{Cohesive Trees} \label{tt2wklsec3}

The main result of this section (Corollary \ref{CTTandWWKL}) is that the $\tcoh$ principle does not imply $\msf{WWKL}$ over $\msf{RCA}_0$.
We decompose $\msf{CTT}^2_k$ into $\wtcoh$ and a principle called the $k$-\treesplit\ principle $k$-$\mathsf{TSP}$  (Definition \ref{tree-split-principle})
whose proof-theoretic strength illustrates   the gap between $\wtcoh$ and $\tcoh$.
In fact this gap can be filled by solution sets which are low relative to   instances of $\wtcoh$:
For any model $\mathfrak{M}_1=(\omega, \mcal{S}_1)$ of $\wtcoh$,
there is a model $\mathfrak{M}_2=(\omega, \mcal{S}_2)$ of $\tcoh$ where $\mcal{S}_2$ is obtained from $\mcal{S}_1$ by adding only sets
that are low relative to members of $\mcal{S}_1$ (Theorem \ref{th3}).
We show that $\wtcoh$ and  $k$-$\mathsf{TSP}$  both admit avoidance of bounded enumeration (Proposition \ref{prop4} and Theorem \ref{th3}).
Combining these results, one concludes that $\tcoh$ admits avoidance of bounded enumeration as well (Corollary \ref{coro103}) and hence does not prove $\msf{WWKL}$.
The decomposition of $\tcoh$ may be viewed as an analog of the decomposition of $\msf{RT}^2_2$ into the cohesive set principle $\msf{COH}$ and the stable  2-coloring principle $\msf{SRT}^2_2$ given in \cite{Cholak2001strength}.
This is discussed in Remark \ref{decomposition}.

\subsection{Enumeration avoidance property of $\wtcoh$}

We follow the approach in Liu \cite{Liu2015Cone},  where  the enumeration avoidance strategy   was used to show that $\mathsf{RT}^2_2$ does not imply $\mathsf{WWKL}$.
We begin with recalling some basic definitions (note that for the purpose of defeating $\mathsf{WWKL}$, avoiding $1$-enumeration is sufficient).

\begin{definition} \label{boundenumeration}
Given a set $S\subseteq 2^{<\omega}$, an $l$-\emph{enumeration} of $S$ is a function $g:\omega\rightarrow \mathsf{Fin}(2^{<\omega})$ such that
$|g(n)|\leq l$ and $g(n)\cap S\cap 2^n\ne\emptyset$ for all $n$.
A \emph{bounded enumeration} of $S$ is an $l$-enumeration of $S$ for some $l\in\omega$.
Given $D\subset\omega$, we say that $S$ admits a $D$-computable $l$-enumeration (resp.~bounded enumeration)
if there is a $D$-computable function that is an $l$-enumeration (resp.~bounded enumeration) of $S$.
\end{definition}

\begin{definition}\label{defavoidance}
A problem $\msf{P}$ admits \emph{avoidance of $l$-enumeration} (resp.~\emph{bounded enumeration}) if for any $D\subset\omega$,
any $S\subseteq 2^{<\omega}$ that does not admit a $D$-computable $l$-enumeration (resp.~bounded enumeration),
and any $D$-computable instance of $\msf P$, there exists a solution $T$ of the instance such that $D\oplus T$ does not compute an
$l$-enumeration (resp.~bounded enumeration) of $S$.

$\msf{P}$ admits \emph{strong avoidance of $l$-enumeration} (resp.~\emph{bounded enumeration}) if for any $D$,
any $S\subseteq 2^{<\omega}$ that does not admit a $D$-computable $l$-enumeration (resp.~bounded enumeration),
and any $\msf{P}$-instance (not necessarily $D$-computable), there exists a solution $T$ of the instance such that $D\oplus T$ does not compute an
$l$-enumeration (resp.~bounded enumeration) of $S$.
\end{definition}

In the following proposition and in Corollary \ref{coro103}, the hypothesis in the ``moreover'' part is redundant   in view of Corollary \ref{strong} to be proved later.

\begin{proposition} \label{prop4}
For each $k$, $\wtcoh$ admits avoidance of $1$-enumeration and bounded enumeration.
Moreover, if $\tt_{k'}^1$ admits strong avoidance of bounded enumeration for all $k'\in\omega$, then so does $\wtcoh$.
\end{proposition}

\begin{proof}
We prove the proposition  for $1$-enumeration.  The proof for bounded enumeration is similar.
The idea is similar to that of  the proof of $\msf{COH}\not\rightarrow \msf{WKL}$ using  Mathias forcing.

Fix an $S\subseteq 2^{<\omega}$ that does not admit a $1$-enumeration computable in $D\subseteq\omega$.
We may assume that $D=\emptyset$ as the argument below relativizes to any set $D$ that does not compute a $1$-enumeration of $S$.
Let $C: [2^{<\omega}]^2\rightarrow k$ be a computable coloring that does not compute a $1$-enumeration of $S$.
We build an infinite perfect subtree $G$ weakly stable for $C$  satisfying the following requirements:
\begin{itemize}
\item $R_e$: For some $m$, either $\Psi_e^G(m)\downarrow\notin (S\cap 2^m)$ or $\Psi_e^G(m)\uparrow$.
\end{itemize}
Define a sequence $(F_e, X_e)_{e\in \omega}$, where $F_e$ is a finite perfect tree,
and $X_e= \bigcup \{U_{\sigma}:\sigma \in \leaf (F_e)\}$, where each $U_\sigma$ is infinite and perfect.
The generic object $G$ will be $\bigcup_e F_e$.  The construction is carried out  recursively in $\emptyset'\oplus S$.

Let $F_0$ be the root $\varepsilon$ and $X_0=2^{<\omega}$.  Suppose that $(F_e,X_e)$ is defined.
Let $F_{e+1}$ be the (canonically  least) finite perfect tree $F\succ F_e$ such that $(F\setminus F_e)\subset X_e$ and
for some $m$ either $\Psi_e^{F}(m)\downarrow\notin S\cap 2^m$ or $\Psi_e^{F'}(m)\uparrow$ for all perfect and finite $F'\succ F$ such that $(F'\setminus F) \subset X_e$.
Such an $m$ exists and can be computed from  $\emptyset'\oplus S$, since otherwise $\Psi_e$ with the recursive oracle $X_e$ will be a $1$-enumeration of $S$.
{\em This ensures that $R_e$ is satisfied}.
To ensure that $\bigcup_e F_e$ is infinite,
observe that for each $e$ there is a $\Psi$ such that $\Psi^F(m)\downarrow=0^m$ if and only if $F$ contains a subtree isomorphic to $2^{e+1}$.
Thus satisfying $R_e$ for all $e$  ensures that $G$ is infinite.

Next,  for each $\sigma\in \leaf(F_{e+1})$, let $U_{\sigma}$ be an infinite recursive  perfect tree extending $\sigma$ such that
$U_{\sigma}\subset X_e$ and for all $\tau_1,\tau_2\in U_{\sigma}$ and for all $\sigma_1\preceq \sigma$, $C(\sigma_1,\tau_1)=C(\sigma_1,\tau_2)$.
The existence of such a $U_{\sigma}$ follows from a (repeated application of) a standard argument regarding density of colors in a perfect tree.
Let $X_{e+1}=\bigcup \{U_{\sigma}: \sigma\in \leaf(F_{e+1})\}$. $X_{e+1}$ is recursive because the coloring $C$ is.
$X_{e+1}$ ensures the weak stability of $G$ up to $F_{e+1}$.
(Note that $\Sigma^0_3$-induction is sufficient to carry out the construction.
We do not know which subsystem of second order arithmetic  is the weakest required).

The proof of the ``moreover" part requires an  additional step.
Suppose that we have obtained $(F_e,X_e)$ where $F_e$ and $X_e$ are as above.
List the leaves of $F_{e}$ from left to right as $\sigma_0,\dots,\sigma_n$ for some $n$.
We will thin the subtree of $X_e$ above each $\sigma_i$, $i\le n$,   cone by cone  so that  the join of the thinned subtrees does not compute any $1$-enumeration of $S$.
Let $V_{-1}=\emptyset$.  Suppose that we have obtained $V_j$ ($-1\leq j<i$) for some $0\leq i\leq n$
such that (1) $V_j\subset X_e$ is an infinite perfect tree extending $\sigma_j$; (2) for all $\sigma'\prec\sigma_j$, $\tau_1,\tau_2\in V_j$,
$C(\sigma',\tau_1)=C(\sigma',\tau_2)$; and (3) $\bigoplus_{-1\leq j<i}V_j$ does not compute any $1$-enumeration of $S$.
Let $|\sigma_i|=k'$ and $U_i=X_e\cap [\sigma_i]^{\preceq}$. Then $C(\sigma',\tau)$, for $\sigma'\preceq \sigma_i$ and $\tau\in U_i$, induces
a $2^{k'}$ coloring on $U_i$.  By the strong avoidance of $\tt^1_{2^{k'}}$,
there is a homogenous solution $V_i\subseteq U_i$ such that $ F_e \bigoplus_{-1\leq j\leq i}V_j$ does not compute a $1$-enumeration of $S$.
Let $Y=\bigoplus_{1\leq i\leq n}V_i$.  Since $\Psi_e^{ F_e\oplus Y}$ does not compute a $1$-enumeration of $S$, there exists an $m$ such that
either for some perfect finite tree $F$, $(F\setminus F_e)\subset Y$ and $\Psi_e^{ F}(m)\downarrow \notin (S\cap 2^m)$,
or for all perfect finite $F\subset Y$, $\Psi^{ F_e\oplus F}(m)\uparrow$.  In the former case, we let $F_{e+1}$ be (the canonical least) such $F$, while
in the latter case we add a split (taken from $Y$) to each of the leaves of $F_e$ to obtain $F_{e+1}$.
In both cases, let $X_{e+1}=\bigcup \{Y\cap [\sigma]^{\preceq}: \sigma\in \leaf(F_{e+1})\}$. Then $(F_{e+1},X_{e+1})$ ensures the satisfaction of $R_e$.

Let $G=\bigcup_{e\in \omega} F_e$. By construction, $G$ is a $\wtcoh$ solution of $C$ and does not compute a $1$-enumeration of $S$. This completes the proof.
\end{proof}

\subsection{Between $\wtcoh$ and $\tcoh$} \label{sectcoh}

We now consider   $\tcoh$.  We first introduce a new combinatorial principle called the $k$-\treesplit\ principle that  serves to fill  the gap between $\wtcoh$ and $\tcoh$.
For a set $B$ of strings, let $B\uhr n = \{\xi\uhr n: \xi \in B\big\}$.

\begin{definition} \label{k-tree-split}
Let $k>0$.  A $k$-\emph{\treesplit\ } on a \treeset\ $T$ is a function $f:T\rightarrow \mathsf{Fin}(k^{<\omega})$
such that  for each $\rho\in T$, if $|\rho|_T=n$ then
\begin{enumerate}
\item  [(i)]  $f(\rho)$ is a nonempty subset of $k^{n}$;

\item [(ii)]
If $\rho$ is not a leaf of $T$, then $
f(\rho)=\bigcup\{f(\rho^+)\uhr n: \mbox{$\rho^+$ is an immediate successor of $\rho$ in $T$}\}.
$
\end{enumerate}
\end{definition}

Item (ii) gives the motivation for   the term  ``\treesplit'': when one moves from $\rho$ to its immediate successors $\rho^+$,
 the set $f(\rho)$ is split  into several (possibly overlapping) subsets, namely the sets $f(\rho^+)\uhr n$.
As an illustrative  example, let $T$ be an infinite perfect tree and let $C:[T]^{2}\to k$ be a $k$-coloring.
For a node $\rho\in T$ with $T$-length $n$, define
\[
f(\rho)=\{\zeta\in k^n:\text{\ there are infinitely many $\tau\succeq \rho$ such that $\zeta=\langle C(\rho\uhr i, \tau)\rangle_{ i<n}$}\}.
\]
It is easy to verify (in any model satisfying  $\Sigma^0_2$-bounding) that $f$ is a $k$-\treesplit\ on $T$.
In our application, $T$ takes the form $[B]^\preceq$ where $B$ is a finite antichain.
In this paper, we assume that every $k$-\treesplit\ is defined on a
nonempty \treeset\ with no leaf.

The following basic properties of a $k$-tree-split  are easy consequences of Definition \ref{k-tree-split}:
\begin{lemma} \label{lem11}
Let $f$ be a $k$-\treesplit\ on $T$ and let $\rho\in T$ be of $T$-length $n$.  Then
\begin{enumerate}
\item [{\rm (i)}] For any  $\rho\preceq \tau$ in $T$, $f(\tau)\uhr n\subseteq f(\rho)$;

\item [{\rm (ii)}]  If $\zeta\not\in f(\rho)$ then for any $\tau\in T$ with $\rho\preceq \tau$, $\zeta\notin f(\tau)\uhr n$.
\end{enumerate}
\end{lemma}

\begin{proof}
(i)  follows from  performing induction on the  set $\{i: |\rho|_T\le i\le |\tau|_T\}$, while (ii) is essentially a re-statement of (i).
\end{proof}

\begin{definition}\label{induced}
Given trees $\h T\subseteq T$ and $k$-tree-splits $\h f, f$ on $\h T$ and $ T$ respectively,
we say that $\h{f}$ is {\it induced} by
$f$ on $\h{T}$  if:
For any $\rho\in \h{T}$ and  $\h{\xi}\in \h{f}(\rho)$,
there exists a $\xi\in f(\rho)$
such that for each $\sigma\in \h{T}$
with $\sigma\preceq \rho$,
$\h{\xi}(|\sigma|_{\h{T}}) =
\xi(|\sigma|_T)$.
In this  case we say that $\h{\xi}$ is induced by
$\xi$ on $\h{T}$.
\end{definition}

\begin{definition} \label{homog_split}
Given a $k$-\treesplit\ $f$ on $T$, a \treeset\ $G\subseteq T$ is  \emph{homogeneous} for $f$
if there exists a $k$-\treesplit\ $f_{G}$ on $G$ induced by $f$ such that
for each $\rho\in G$, $f_{G}(\rho)$ is a singleton.

\end{definition}
  Thus $f_G$ is a single-valued function on $G$ preserving length and extension, i.e.~$|\sigma|_G=|f_G(\sigma)|$  for $\sigma\in G$, and if $ \sigma\prec\tau$ in $G$,  then $f_G(\sigma)\prec f_G(\tau)$. We call $f_G$ a {\it homomorphism} for short.
  We also say in this case that $f_G$ witnesses $G$ to be homogeneous for $f$.
For a $k$-\treesplit\ $f$ on $T$, a homogenous tree $G$ is a subtree of $T$ such that $f_{G}(\rho)$ is a color vector $\zeta$ with the property that for any  $\tau\succeq \rho$ in $G$,
$\zeta\preceq f_G(\tau)$.  Since $f_G$ is a \treesplit , for any $\rho,\tau\in G$, if $\rho\preceq \tau$ then $f_G(\rho)\preceq f_G(\tau)$.
Note that
given a $k$-tree-split $f$ on a tree $T$, it makes sense to say that a finite $F\subset T$ is  homogeneous for $f$, namely
there is a function $f_F$ induced by $f$ and defined on $F$, such that
for each $\rho\in F$, $f(\rho)$ is a singleton.


\begin{lemma} \label{lem3.9vii}
Let  $T$ be an infinite  tree without leaves.
\begin{enumerate}[(i)]
\item [{\rm (i)}] The collection $Q=\{f: f$ is a $k$-\treesplit\ on $T\}$ is a $\Pi^{0,T}_1$-class.
\item [{\rm (ii)}] Let $f$  be a $k$-tree-split on $T$. Whether a finite \treeset\ $F$ is homogeneous for $f$ is decidable uniformly in $f,F,T$.

\end{enumerate}
\end{lemma}

\begin{proof}
Let $T_n=\{\rho: \rho\in T\wedge |\rho|_T\le n\}$.
Let $U_n$ be the set of (codes of) finite functions from  $T_n$  to $\mathsf{Fin}(k^{<\omega})$ satisfying conditions (i) and (ii) in Definition \ref{k-tree-split} (restricted to $T_n$). Let $U=\bigcup_n U_n$, and
order its members  by  (functional) extension.
Then   $U$ with the ordering relation   is a finite branching recursively bounded  recursive tree and the set of all infinite paths in $U$ is exactly $Q$.

The proof of (ii) is immediate.
\end{proof}

\begin{definition} \label{wstable-tree-split}
A $k$-\treesplit\ $f$ on $T$ is \emph{weakly stable}
if
\[
(\forall \rho\in T)(\exists d)(\forall \rho'\in T\cap [\rho]^\preceq)[|\rho'|_T\geq d\rightarrow (\exists\zeta\in f(\rho))(f(\rho')\uhr |\rho|_T=\{\zeta\})].
\]
\end{definition}
Notice that $f(\rho')\uhr |\rho|_T=\{\zeta\}$ says that $\zeta$ is the common initial segment of every member of $f(\rho')$.
By Lemma \ref{lem11}, $\zeta$ is also  the common initial segment of $f(\tau)$ for every $\tau$ extending $\rho'$.
Thus, above level $d$, the tree $T\cap [\rho]^{\preceq}$ is partitioned into disjoint cones each of which ``shares'' a single initial segment $\zeta\in f(\rho)$
in the above sense, though different cones may share a different $\zeta$.
The following proposition  gives the  intuition behind  the notion of a $k$-\treesplit.

\begin{proposition} \label{prop2}
For any \wstable\ coloring $C:[2^{<\omega}]^2\rightarrow k$, there exists a weakly stable
$k$-\treesplit\ $f_C:2^{<\omega}\rightarrow \mathsf{Fin}(k^{<\omega})$ with $f_C\leq_T C'$ such that
$C\restriction G$ is stable for any infinite perfect tree  $G$ that is homogeneous for $f_C$.
\end{proposition}

\begin{proof} For any $\rho\in 2^{<\omega}$   define
$$
f_C(\rho) =\big\{\zeta\in k^{|\rho|}: (\exists \rho'\succeq\rho)(\forall \tau\succeq\rho')(\forall i<|\rho|)[C(\rho\uhr i,\tau)=\zeta(i)]\big\},
$$
in other words, $\zeta \in f_C(\rho)$ if and only if on a cone above $\rho$ with base $\rho'$,
the color vector $\langle C(\rho\uhr i, \tau)\rangle_{ i<|\rho|}$ is $\zeta$ for all $\tau$ in the cone.
Thus it is immediate that $f_C$ satisfies (ii) of Definition \ref{k-tree-split}.
The weak stability of $C$ implies that $f_C(\rho)\neq \emptyset$, $f_C$ is a weakly stable $k$-tree split and $f_C\leq_T C'$.

Suppose $G$ is homogeneous for $f_C$ with witness $f_G$.
Fix a $\rho\in G$ and suppose  $f_G(\rho)=\zeta$. Since both the coloring $C$ and the $k$-\treesplit\ $f_C$ are weakly stable,
there is a $d\in \omega$ such that for all $\sigma\succeq \rho$ with $|\sigma|=d$, for all $\tau\succ\sigma$, we have
\begin{enumerate}
  \item [(i)] $C(\rho,\tau)=i_{\sigma}$, for some $i_{\sigma}<k$ which does not depend on $\tau$.
  \item [(ii)] For all $\xi\in f_C(\tau)$,  there is a cone above $\tau$ such that for every $\tau'$ in the
  cone, $C(\rho,\tau')=\xi(|\rho|)$.  By (i), we have  $\xi(|\rho|)=i_{\sigma}$.
\end{enumerate}
Thus  for all $\tau\in G\cap [\rho]^\preceq$,
 the value of
 $f_G(\tau)$ at input  $|\rho|_G $, written $f_G(\tau)(|\rho|_G)$,
   is equal to  $ \zeta(|\rho|_G)$.
   Suppose $\sigma\succeq \rho$ with $|\sigma| = d$,
  $\tau\in G\cap [\sigma]^\preceq$, and
 $f_G(\tau)$ is induced by $\xi \in f_C(\tau)$
 on $G$. Then we have:
\[
\zeta(|\rho|_G) = f_G(\tau)(|\rho|_G)=\xi(|\rho|)=i_{\sigma} = C(\tau,\rho).
\]
In other words, for every $\sigma\succeq \rho$  with $|\sigma|= d$,
if $G\cap [\sigma]^\preceq\ne\emptyset$, then $C(\tau,\rho) = \zeta(|\rho|_G)$
for all $\tau\in G\cap [\sigma]^\preceq$. Hence $C$ is stable on $G$.
\end{proof}

The following proposition serves as  a converse to Proposition \ref{prop2},
showing that $k$-\treesplit\ is key to a decomposition of $\tcoh$.

\begin{proposition} \label{propcjsctt}
For any weakly stable $\Delta_2^0$ $k$-\treesplit\ $f$ on $2^{<\omega}$, there exists a \wstable\ computable coloring
$C: [2^{<\omega}]^2 \rightarrow k$ such that every $\tcoh$ solution of $C$ is homogeneous for $f$.
\end{proposition}

\begin{proof}
 Fix a recursive approximation $f[s]$ of $f$.
By Lemma \ref{lem11}, without loss of generality, we assume $f(\rho')[s]\uhr |\rho|\subseteq f(\rho)[s]$
for all $\rho'\succeq \rho$ and stages $s$.
For each node $\sigma'$ with $\sigma\prec \sigma'\preceq \rho$, compute $f(\sigma')[|\rho|]$, and  find the  $\sigma'$ of shortest length such that $f(\sigma')[|\rho|]$ is a singleton $\{\xi\}$ if it exists.
Define
\[
C(\sigma,\rho)=\left\{
  \begin{array}{ll}
    \xi(|\sigma|) &\text{if such a } \xi \text{ exists} \\
    0 & \hbox{otherwise.}
  \end{array}
\right.
\]
Since $f$ is weakly stable, $C$ is \wstable.

 Now let $T$ be a $\tcoh$ solution of $C$.
 Let $\rho\in T$ be of $T$-length $n$.
 Apply the stability on $T$ to define
\[
\t{C}(\rho) = \lim\limits_{\rho'\in [\rho]^\preceq\cap T,|\rho'|\rightarrow\infty} C(\rho,\rho').
\]
Let $\rho_0\prec\rho_1\prec\cdots\prec \rho_{n-1}=\rho$ be the predecessors of $\rho$ in $T$, and
define $f_T(\rho) =\t{C}(\rho_0)^\smallfrown\cdots^\smallfrown\t{C}(\rho_{n-1})$.
By the definition of $f_T$, it is clear that $f_T(\rho)$ is a singleton and $|\rho|_T=|f_T(\rho)|$.
To show that $T$ is homogeneous for $f$,
it  remains to show that $f_T$ is induced by $f$, namely, if $f_T(\rho)=\zeta$, then there exists a
$\xi\in f(\rho)$ such that for all $i\leq n$, $\zeta(i)=\xi(|\rho_i|)$.
Choose $d$ sufficiently large such that for any $\rho'\succ \rho$ with $|\rho'|_T\geq d$ we have
(1) For all $i< |\rho|$, $C(\rho_i,\rho')$ is a constant independent of $\rho'$ (here we use the fact  that $T$ is a solution of $\tcoh$);
and (2) $f(\rho')[|\rho'|]\uhr n=\{\xi\}$ for some $\xi$ (here we use the fact  that $f$ is a weakly stable \treesplit ). Thus,
\[
\zeta(i)=f_T(\rho_i)=\t{C}(\rho_i)=C(\rho_i,\rho')=\xi(|\rho_i|).
\]
\end{proof}

The above propositions provide the ingredients for   a decomposition of $\tcoh$.  We first introduce a principle based  on $k$-tree-splits.

\begin{definition} \label{tree-split-principle}
Let $k\in \omega$.  The {\em $k$-\treesplit\ principle} ($k$-$\mathsf{TSP}$) states that for every infinite perfect tree $T\subseteq 2^{<\omega}$,
and  $k$-\treesplit\ $f$ on $T$ which is $\Delta^0_2$ over $T$, there exists an infinite perfect tree $G\subset T$  homogenous for $f$
with a witness function $f_G$ that is $\Delta^0_2$ over $T$.
\end{definition}

Let $B\Sigma^0_2$ denote the $\Sigma^0_2$-bounding induction scheme, i.e.~a model  of $\mathsf{RCA}_0$ satisfies $B\Sigma^0_2$ if every $\Sigma^0_2$-definable function maps a finite set (in the sense of the model) to a finite set.

\begin{corollary}\label{decomposition}
Over the base theory $\mathsf{RCA}_0+B\Sigma^0_2$,
\[
\mathsf{CTT}^2_k=\mathsf{wCTT}^2_k+k\text{-}\mathsf{TSP}.
\]
\end{corollary}

\begin{proof}
Note that there is a difference, though not immediately apparent, between a weakly stable coloring $C:[2^{<\omega}]^2\rightarrow k$ and a weakly stable $k$-tree-split $f$, namely for each $\sigma$,
the existence of a $d$ that guarantees the weak stability property of $C$ to hold for all $\sigma'\preceq\sigma$ requires an appeal to $B\Sigma^0_2$,
whereas for $f$ the existence of a corresponding $d$ that applies to all $\sigma'\preceq\sigma$ is part of the definition.
Thus with this in mind,
Propositions \ref{prop2} and \ref{propcjsctt} are provable in $\mathsf{RCA}_0+B\Sigma^0_2$. The corollary follows from these propositions.
\end{proof}

\begin{remark}
One may view Coollary \ref{decomposition} as an analog of the decomposition of $\mathsf{RT}^2_2$ into the sum of $\mathsf{COH}$ and $\mathsf{SRT}^2_2$ given in \cite{Cholak2001strength}.
This decomposition assumes the following form: A principle
$\mathsf{P}$ is the sum of two principles denoted $\msf{C}$ (the ``cohesive part'') and $\msf{S}$ (the ``stable part''),  each of which is a consequence of $\msf P$ over a base system.
Furthermore, $\msf S$ is a ``$\Delta^0_2$-version'' of $\msf P$ (relative to some parameters).
A problem that is an instance of $\msf{P}$ is solved by first applying $\msf{C}$ to reduce it to an instance of a $\Delta^0_2$-version of $\msf P$ (relative to some parameters),
and then applying $\msf{S}$ to obtain a solution. Hirschfeldt and Shore \cite{Hirschfeldt2007Combinatorial} showed that over the system  $\mathsf{RCA}_0+B\Sigma^0_2$,
$\mathsf{COH}$ is equivalent to the principle called cohesive Ramsey's theorem for pairs $\mathsf{CRT}^2_2$.\footnote{$\mathsf{CRT}^2_2$ states that for every 2-coloring $C$ of $[\mathbb{N}]^2$,
there is an infinite set $X$  and a function $f: [\mathbb{N}]^2\rightarrow 2$ such that  $\text{lim}_{y\in X} f(x,y)$ exists for every $x\in X$.}
$\mathsf{CTT}^2_2$ is the tree version of $\mathsf{CRT}^2_2$ and Corollary \ref{decomposition} is a decomposition of $\tcoh$  into two components.
For $\tcoh$, an instance of the problem concerning 2-coloring of $2^{<\omega}$ is first reduced to a 2-coloring that is weakly stable on an isomorphic subtree by an application of $\mathsf{w}\tcoh$,
and then to a solution of the problem by an application of $k$-$\mathsf{TSP}$.
\end{remark}

\subsection{$k$-\treesplit\  and avoiding bounded enumeration}
This subsection is devoted to proving Theorem \ref{th3}, which is a technical result to be used to demonstrate $\mathsf{CTT}^2_k\not\rightarrow \mathsf{WWKL}$.
Theorem \ref{th3} says further that by  adding an appropriate set that is  low (and of non-\PA-degree) relative to a solution of $\wtcoh$,
one can obtain a solution  of $\tcoh$ that admits avoidance of $1$-enumeration and bounded enumeration.

We first introduce some notions related to $k$-\treesplit s and show properties of these notions  which are used in the proof of Theorem \ref{th3}.

\begin{definition}
Given a $k$-\treesplit\ $f$ on $T$,
$f$ is \emph{\concised\ } if
for every $\rho\in T$, every $\zeta\in f(\rho)$,
there exists a $\rho'\in T\cap [ \rho]^\preceq$ such that
$\zeta\in f(\rho'')\uhr |\rho|_T$ for all
$\rho''\in T\cap  [\rho']^\preceq$.
\end{definition}
We say that a $k$-\treesplit\ $\t{f}$ on $T$  is a \emph{\conciseversion\ } of $f$ if
$\t{f}$ is \concised\ and $\t{f}(\rho)\subseteq f(\rho)$ for all $\rho\in T$.
In other words, a \conciseversion\ of $f$ collects $\zeta\in f(\rho)$ which ``appears'' on a cone above $\rho$.

\begin{lemma}\label{lem3.9ii}
Every $k$-\treesplit\  on a tree $T$ admits a \conciseversion.
\end{lemma}
\begin{proof}
Fix a $k$-\treesplit\ $f$ on $T$. Let
$$\t{f}(\rho) = \big\{
\zeta\in f(\rho): (\exists \rho'\in T\cap [\rho]^\preceq)
(\forall \rho''\in T\cap [\rho']^\preceq)
[\zeta\in f(\rho'')\uhr |\rho|_T]
\big\}.$$
We first prove that $\t{f}$ is a  $k$-\treesplit \ on $T$.
Suppose that $\zeta\in \t{f}(\rho)$.
Then there is a cone with base $\rho'$ above $\rho$ such that for any
$\rho''\succeq \rho'$, $\zeta\in f(\rho'')\uhr |\rho|_T$.
This $\rho'$ must extend one of the successors of $\rho$, say $\rho^+$.
Thus $\zeta\in \t{f}(\rho^+)\uhr |\rho|_T$.
On the other hand, if $\zeta\notin \t{f}(\rho)$,
 then for any cone above $\rho$, there is a $\rho''$ in the cone such that
  $\zeta\notin f(\rho'')\uhr |\rho|_T$, which implies
  $\zeta^+ \notin f(\rho'')\uhr |\rho^+|_T$
  (where $\zeta^+$ is an arbitrary immediate successor
  of $\zeta$). Thus
$\zeta^+\notin \t{f}(\rho^+)$
for every immediate successor $\zeta^+$
of $\zeta$, which means $\zeta\notin \t{f}(\rho^+)\uhr |\rho|_T$.

Next we show that $\t{f}(\rho)$ is nonempty.  For  $\tau\succeq \rho$ let $D_{\tau}=f(\tau)\uhr |\rho|_T$.
Suppose that there is a $\zeta\in f(\rho)$ such that for some cone with base $\rho_1\succeq \rho$, $\zeta\notin f(\rho_1)\uhr |\rho|_T$.
Then $D_{\rho_1}\subsetneq D_{\rho}$.  Now replace $\rho$ by $\rho_1$ and repeat the process to obtain $\rho_2\succ\rho_1$ with $D_{\rho_2}\subsetneq D_{\rho_1}$.
This process must end at some $\rho^*\succ\rho$.
Any $\zeta\in f(\rho^*)\uhr |\rho|_T$ will be in $f(\rho'')\uhr |\rho|_T$ for every $\rho''$ in the cone above $\rho^*$. Thus $\zeta$ is in $\t{f}(\rho)$ and
hence $\t{f}(\rho)$ is nonempty.
This verifies that $\t{f}$ is a $k$-\treesplit\ on $T$.

Now we prove that $\t{f}$ is \concised.
Let $\zeta\in \t f(\rho)$. Then for some cone with base $\rho'\succeq \rho$, for all $\rho''\succeq \rho'$ in the cone,
$\zeta\in f(\rho'')\uhr |\rho|_T$.
We show, by contradiction,
 that $\zeta\in \t{f}(\rho'')\uhr |\rho|_T$ for all $\rho''\in T\cap [\rho']^\preceq$.
Suppose otherwise, i.e.~$\zeta\notin \t{f}(\rho^*)\uhr |\rho|_T$ for some
$\rho^*\in T\cap [\rho']^\preceq$.
This means $[\zeta]^\preceq\cap k^{|\rho^*|_T}\cap \t{f}(\rho^*) = \emptyset$.
By the definition of $\t{f}$, there exists a $\sigma\in T\cap [\rho^*]^\preceq$
such that
$ [\zeta]^\preceq\cap k^{|\rho^*|_T}\cap f(\sigma)\uhr |\rho^*|_T = \emptyset$,
which means $\zeta\notin f(\sigma)\uhr |\rho|_T$.
 This contradicts the choice of $\rho'$ and the fact that $\sigma\in T\cap [\rho']^\preceq$.
\end{proof}

\begin{definition}
Let $f_0$ and $f_1$ be two $k$-\treesplit s on $T$.  Define the \emph{\cross}  $f_0\otimes f_1$ of $f_0, f_1$  as follows:
For every $\rho$,
\[
(f_0\otimes f_1)(\rho) = \big\{(\zeta_0,\zeta_1)\in f_0(\rho)\times f_1(\rho):
(\exists\rho'\in T\cap  [\rho]^\preceq)(\forall \rho''\in T\cap[ \rho']^\preceq)
\forall i\in\{0,1\}[ \zeta_i\in f_i(\rho'')\uhr |\rho|_T]\big\}.
\]
\end{definition}

In other words, $(f_0\otimes f_1)(\rho)$ collects the pairs $(\zeta_0,\zeta_1)$ such that on a cone above $\rho$, every $\rho''$ in the cone
has the property that $\zeta_0$ and $\zeta_1$ are initial segments of $f_0(\rho'')$ and $f_1(\rho'')$ respectively.
One may view  $f_0\otimes f_1$  as a common \conciseversion\ of  $f_0$ and $f_1$.

The definition of a tree being homogeneous for a $k$-tree-split can be generalized in the obvious way  to it being homogeneous for a cross product of $k$-tree-splits.  This gives us (ii) in the following lemma.

\begin{lemma} \label{lem3.9iv}
If $f_0,\dots, f_n$ are  $k$-\treesplit s on a \twobranching\ \treeset\ $T$, then
\begin{enumerate}
\item [{\rm (i)}] $f_0\otimes f_1$ is a \concised\ $k^2$-\treesplit \ on $T$;
\item [{\rm (ii)}] If $G$ is a tree homogeneous for $f_0\otimes f_1$, then $G$ is homogeneous for  $f_0$ and $f_1$;

\end{enumerate}
\end{lemma}

\begin{proof}
The proof of (i) is similar to the proof of Lemma \ref{lem3.9ii}.

(ii):  Let $h$ be a witness for $G$ being homogeneous for $f_0\otimes f_1$.
Define $h_0:G\to k^{<\omega}$ by $h_0(\rho)=\zeta_0$ if for some $\zeta_1$,
$(\zeta_0,\zeta_1)\in h(\rho)$.  Then $h_0$ is a witness for $G$ being homogenous for $f_0$. A
similar argument applies to $f_1$.

\end{proof}

The next  lemma motivates
      the notion of a refined $k$-\treesplit .
Working with \concised \ \treesplit s allows one to extend  a finite homogeneous tree to a larger one while preserving homogeneity.

\begin{lemma} \label{lem3.9x}
Let $f$ be a $k$-\treesplit\ on a \treeset\ $T$
and suppose $F\subset T$ is a finite \treeset\ homogeneous for
some \conciseversion\ of $f$.
\begin{enumerate}[(i)]
\item [{\rm (i)}]   There exist a
finite antichain $B\subseteq T$ with $B\cap [\sigma]^\prec\ne\emptyset$
for each $\sigma\in \ell(F)$, and a  $k$-\treesplit\ $\h{f}$ on  $T\cap [B]^\preceq$
such that for every \treeset\ $G\subseteq T\cap [B]^\preceq$  homogeneous for $\h{f}$, $F\cup G$ is homogeneous for $f$.
\item [{\rm (ii)}] If $f$ is of the  form
$f_0\otimes\cdots \otimes f_n$ where $f_0$ is weakly stable,
then $\h{f}$ is  of the  form $\h{f}_0\otimes\cdots \otimes \h{f}_n$
where $\h{f}_0$ is weakly stable, and
$B$, $\h{f}_0\oplus (T\oplus \h{f}_1\oplus\cdots\oplus \h{f}_n)'$
are  computable uniformly from $F, f_0\oplus (T\oplus f_1\oplus\cdots\oplus f_n)'$.
\end{enumerate}
\end{lemma}

\begin{proof}
(i). Suppose $h_F: F\rightarrow k^{<\omega}$ witnesses  $F$ being homogeneous for
a \conciseversion\ $\t{f}$ of $f$.
For every $\sigma\in\leaf(F)$, let $\zeta_{\sigma}=h_F(\sigma)$ which is induced by some $\t{\zeta}_{\sigma}\in \t{f}(\sigma)$
on $F$.
By the definition of a  \conciseversion , there is a cone in $T$ with (the canonically least) base $\rho_{\sigma}\succ\sigma$ such that for any $\rho'$ in the cone, any $\zeta'\in f(\rho')$,
$\zeta'$ has $\t{\zeta}_\sigma$ as an initial segment.
Let $B=\{\rho_{\sigma}:\sigma\in \leaf(F)\}$.

For $\sigma\in \ell(F)$,  $\rho'\in T$ with $\rho'\succeq \rho_\sigma$ and ${\xi}\in f(\rho')\cap [\t{\zeta}_\sigma]^\preceq$, define
\[
\eta_{\xi}=
                \xi\uhr [|\rho_\sigma|_T, |\rho'|_T],
\]
where $ [|\rho_\sigma|_T, |\rho'|_T]=\{s: |\rho_\sigma|_T\le s\le |\rho'|_T\}$.
The informal idea is as follows: Given  $\xi\in f(\rho')$, if  $\xi$ does not extend $\t{\zeta}_{\sigma}$, ignore it.
Otherwise, by the choice of $\rho_{\sigma}$,   we divide $\xi$ into three sections:
The first section is $\t{\zeta}_{\sigma}$; the second  is  the section  $\xi\restriction \{s:|\sigma|_T<s<|\rho_\sigma|_T\}$, and the rest of $\xi$ constitutes the third section.
We then ignore the second section and take the third  as $\eta_{\xi}$.  For $\sigma\in\ell (F)$ and $\rho'\succeq\rho_\sigma$, let
$$
\h{f}(\rho') = \big\{\eta_{\xi}:
\xi\in f(\rho')\cap [\t{\zeta}_\sigma]^\preceq
\big\}.
$$

Clearly  $\h{f}$ is a $k$-\treesplit\  \ on $T\cap [B]^\preceq $ by the choice of $\rho_{\sigma}$ and the fact that $f$ is a  $k$-\treesplit \ on $T$.
Suppose $G\subseteq T\cap [B]^\preceq $ is homogeneous for $\h{f}$ with $h_G$ as a  witnesses.
Then  $G\cap F=\emptyset$ since $\rho_\sigma\succ\sigma$.
To see that $F\cup G$ is homogeneous for $f$, we define its witness function $h_{F\cup G}:F\cup G\rightarrow k^{<\omega}$ as follows.
For each  $\tau\in F\cup G$, if $\tau\in G$ then there exists a $\sigma\in \leaf(F)$ such that $\tau\succeq\rho_\sigma$. Let
$h_{F\cup G}(\tau) = h_F(\sigma)^{\smallfrown}h_{G}(\tau)$.
If $\tau\in F$, define $h_{F\cup G}(\tau) = h_F(\tau)$.
Since $F$ and $G$ are  homogenous witnessed by $h_F$ and $h_{  G}$ respectively,
 we have $h_{F\cup G}$ to be a homomorphism from $F\cup G$ to $k^{<\omega}$ that preserves length.
By the choice of $\rho_{\sigma}$ and the definition of $\h{f}$, we know that for every node $\tau\in G$ which is in the cone above $\rho_{\sigma}$,
there exists a $\xi\in f(\tau)\cap [\t{\zeta}_\sigma]^\preceq$
 such that $h_G(\tau)$ is induced by $ \eta_\xi$ on $G$.
It follows that  $h_{F\cup G}(\tau) = \zeta_\sigma^\smallfrown h_G(\tau)$ is induced by $\xi\in f(\tau)$
on $F\cup G$. This proves (i).

(ii). Given a  homomorphism
 $h:F\rightarrow k^{<\omega}$, an antichain $B\subseteq T$ with
 $B\cap [\tau]^\prec\ne\emptyset$ for each $\tau\in \ell(F)$, we can decide
 uniformly in $f_0\oplus (T\oplus f_1\oplus\cdots\oplus f_n)'$
 whether  for each $\sigma\in F$, $h(\sigma)$ is induced by
 some $\t{\zeta}_\sigma\ =
 (\t{\zeta}_{\sigma,0},\dots,\t{\zeta}_{\sigma,n})\in f(\sigma)$ with the following properties:
 \begin{itemize}
 \item For
 $1\leq m\leq n$,
 $\tau\in\ell(F)$ and
 $\rho\in T\cap [\tau]^\preceq\cap [B]^\preceq$, one has
 $\t{\zeta}_{\tau,m}\in f_m(\rho)\uhr |\tau|_T$;

 \item For  $\tau\in\ell(F)$ and
  $\rho\in [\tau]^\preceq\cap B$,
 one has $f_0(\rho)\uhr |\tau|_T = \{\t{\zeta}_{\tau,0}\}$.
 \end{itemize}

Now if $F$ is homogeneous for some \conciseversion\ of $f$, then such
 $h,B$ must exist. Thus $B$ is computable uniformly in
 $F, f_0\oplus (T\oplus f_1\oplus\cdots\oplus f_n)'$. Similarly as in the proof of
 (i), we define, for  $0\leq m\leq n$,
 $$
 \h{f}_m(\rho') = \big\{
 \xi\uhr[|\rho|_T, |\rho'|_T]:
 \text{ for some }\tau\in\ell(F), \rho\in B\cap [\tau]^\preceq,
 \rho'\succeq\rho\text{ and }
 \xi\in f_m(\rho')\cap [\t{\zeta}_{\tau,m}]^\preceq
 \big\}.
 $$
 Note that weak stability of $f_0$ carries over to $\h f_0$.
 Then  $\h{f} = \h{f}_0\otimes\cdots\otimes \h{f}_n$ is the desired
 $k$-\treesplit\ on $T$.
\end{proof}

\begin{lemma} \label{lem3.9vi}
Let  $f$ be a $k$-tree-split on a \twobranching\ \treeset\ $T$
and fix $n$. Then
\begin{enumerate}[(i)]
\item [{\rm (i)}]  There exists
a finite perfect tree $F\subseteq T$   such that $|F|>n$
 and $F$ is homogeneous for some \conciseversion\ of $f$.

 \item [{\rm (ii)}] If  $f$ is of the  form $f_0\otimes \cdots \otimes f_n$
 where $f_0$ is weakly stable, then the set
  $$\big\{F': F'\text{ is a finite perfect tree and homogeneous for
 some \conciseversion\ of }f\big\}
 $$
 is c.e.~uniformly in
$f_0\oplus (T\oplus f_1\oplus\cdots\oplus f_m)'$.
 Thus such an $F$ is computable uniformly in $f_0\oplus (T\oplus f_1\oplus\cdots\oplus f_m)'$.

 \end{enumerate}
\end{lemma}

\begin{proof}
Note that  any singleton node  is homogeneous for any \conciseversion\ of
$f$. Then one obtains $F$ by  repeatedly applying Lemma \ref{lem3.9x}.
The proof of item (ii) is the same as that of Lemma \ref{lem3.9x}  (ii).
\end{proof}

We now show that $k$-$\mathsf{TSP}$ admits strong avoidance of $1$-enumeration and bounded enumeration.

\begin{theorem} \label{th3}
Let $D\subseteq \omega$. Assume that $S\subseteq 2^{<\omega}$ does not admit a $D$-computable $1$-enumeration (resp.~bounded enumeration).
For any $k^*\in\omega$ and $k^*$-\treesplit\ $f^*$ on $2^{<\omega}$,
there exists an infinite perfect \treeset\ $G$   homogeneous for  $f^*$
such that $S$ does not admit a $D\oplus G$ computable $1$-enumeration (resp.~bounded enumeration).
Moreover, if $f^*\leq_T D'$ is weakly stable and $S$ is $\Pi^0_1$ in $D$, then $G$ can be chosen so that $D\oplus G$ is low relative to $D$.
\end{theorem}

\begin{proof} Fix a $k^*$-tree-split $f^*$ on $2^{<\omega}$.
We prove the theorem for the case of avoiding  $1$-enumeration.
The proof for avoidance of bounded enumeration is similar.
%

We apply the idea of Mathias forcing again and build a sequence of tuples $(F_e,B_e, f_e)_{e\in \omega}$ with the following properties:
\begin{itemize}
\item   $F_e$ is a finite perfect tree;
\item $F_{e+1}\succeq F_e$;
\item $B_e$ is an antichain  such that $B_e \subset [\ell(F_e)]^\preceq$
 and every leaf of $F_e$ has an extension in $B_e$;
\item  $f_e$ is a $k$-\treesplit \ on $[B_e]^\preceq$ for some $k$.
\end{itemize}

We ensure that for every $G\subseteq [B_e]^\preceq$ homogeneous for $f_{e+1}$, $F_e\cup G$  is homogeneous
 for $f_e$, and $F_{e+1}\setminus F_e$ is homogeneous for $f_e$.
Let $(F_0,B_0,f_0) = (\emptyset, \{\varepsilon\}, f^*)$.
The generic object $G^*$ will be $\bigcup_e F_e$. Clearly
$G^*$  is
 homogeneous for $f^*$.  We will  make $G^*$
satisfy the following  requirements:
\begin{itemize}
  \item $\mcal{P}_e$: $|F_e|>e$. (This is to ensure that $G^*$ is infinite and perfect.)
\end{itemize}
\begin{itemize}
\item $\mcal{R}_e$: For some $m$, either $\Psi_e^{D\oplus G^*}(m)\downarrow\notin S\cap 2^m$ or $\Psi^{D\oplus G^*}(m)\uparrow$.
(This is to make $D\oplus G^*$ avoid computing $1$-enumeration of $S$.)
\end{itemize}

 Suppose that $(F_e,B_e,f_e)$ is defined.
We   construct $(F_{e+1},B_{e+1},f_{e+1})$ that
forces $\mcal{R}_e$. For each $n$ and $V\subseteq 2^{n}$,
consider the following set of $k$-\treesplit s  on $[B_e]^\preceq$:
\begin{eqnarray*}
Q^n_V &=& \big\{f: f \text{\ is a $k$-\treesplit \ on $[B_e]^\preceq$,} \\
&&\hspace{1cm} \text{and for every finite perfect \treeset\ }F'\succeq F_e \text{\ such that  $F'\setminus F_e\subset  [B_e]^\preceq$ and homogeneous}\\
&&\hspace{1cm} \text{for $f$, if $\Psi_e^{D\oplus F'}(n)\downarrow$ then\ } \Psi_e^{D\oplus F'}(n)\in V \big\}.
\end{eqnarray*}
It follows from Lemma \ref{lem3.9vii}  that $Q^n_V$ is a $\Pi_1^{0,D}$-class.
We consider three cases.

\medskip

\noindent\textbf{Case 1}. For some $n\in\omega$, $Q^n_{S\cap 2^n}=\emptyset$.

Let  $\t{f}$ be a \conciseversion\ of $f_e$ which exists by Lemma \ref{lem3.9ii}.
Suppose $F$ witnesses $\t{f}\notin Q^n_{S\cap 2^n}$, i.e.~$F\succeq F_e$ is finite  and perfect,
$F\setminus F_e \subseteq [B_e]^\preceq$, $F\setminus F_e$ is homogeneous for $\t{f}$,
$\Psi_e^{D\oplus F}(n)\downarrow$ and $\Psi_e^{D\oplus F}(n)\notin S\cap 2^n$.
Let $F_{e+1}=$(the canonically least such) $F$.
Apply Lemma \ref{lem3.9x} to obtain $B_{e+1}$ and  a $k$-\treesplit\ $f_{e+1}$ on $[B_{e+1}]^\preceq$ such that
\begin{enumerate}
\item [(1)] For every $\sigma\in \leaf(F_{e+1})$ there is a $\rho \in B_{e+1}$ with $\rho\succeq \sigma$;
\item [(2)] For every \treeset\ $G\subseteq [B_{e+1}]^\preceq$ homogeneous for $f_{e+1}$, $F_{e+1}\cup G$ is homogeneous for $f_e$.
\end{enumerate}
Thus the condition $(F_{e+1},B_{e+1},f_{e+1})$ satisfies $\mcal{R}_e$ (by fulfilling the $\Sigma^0_1$-clause).

\medskip

\noindent\textbf{Case 2}. There exist $n\in\omega$ and a family of finite sets $V_0,\dots,V_{r-1}\subseteq 2^{n}$ such that
$\bigcap_{0\leq m\leq r-1} V_m =\emptyset$ and $Q^n_{V_m}\ne\emptyset$ for each $0\leq m\leq r-1$.

For $0\le m\le r-1$, select a $g_m\in Q^n_{V_m}$. Let $\h{f} =f_e\otimes g_0\otimes\cdots\otimes g_{r-1}$.
By Lemma \ref{lem3.9iv} (i), $\h{f}$ is a $k^{r+1}$-\treesplit.
Fix a perfect \treeset\ $G\succeq F_e$ such that $G\setminus F_e\subseteq [B_e]^\preceq$.
By Lemma \ref{lem3.9iv} (ii),  if  $G\succ F_e$ and $G\setminus F_e$ is homogeneous for $\h{f}$,
then $G\setminus F_e$ is homogeneous for $g_m$ for each $m\leq r-1$.
Now if  $G\setminus F$ is homogeneous for each $g_m$, then whenever $\Psi_e^{D\oplus G}(n)\downarrow$,
it is a member of $\bigcap_{0\leq m\leq r-1} V_m$ so that $\bigcap_{0\leq m\leq r-1} V_m\ne \emptyset$ which contradicts the assumption.
Thus $\Psi_e^{D\oplus G}$ is not total if  $G\setminus F_e$ is homogeneous for $\h{f}$.
Hence let $F_{e+1}=F_e$, $B_{e+1}=B_e$ and $f_{e+1}= f_e\otimes g_0\otimes\cdots\otimes g_{r-1}$.
The condition $(F_{e+1},B_{e+1},f_{e+1})$ satisfies $\mcal{R}_e$ (by fulfilling the $\Pi^0_1$-clause).

\medskip

\noindent\textbf{Case 3}. Otherwise.

We show that there exists a $D$-computable $1$-enumeration $g$ of $S$ which contradicts the hypothesis.
For every $n$ find (the least) stage $s$ and (the corresponding least string) $\gamma_n$ in the computation of $Q^n_V[s]$  such that
for every $V\subseteq 2^n$, either $Q^n_V[s]=\emptyset$ or $\gamma_n\in V$. Then define $g(n) = \{\gamma_n\}$.
Since $Q^n_V $ is $\Pi_1^{0,D}$, ``$Q^n_V=\emptyset$'' is $\Sigma^{0,D}_1$, hence $g\leq_T D$.
By the failure of Case 2, there is some $\gamma\in \bigcap\{V\subseteq 2^n: Q^n_V\ne\emptyset\}$.
By the failure of Case 1, $S\cap 2^n\in \bigcap\{V\subseteq 2^n: Q^n_V[s]\ne\emptyset\}$.
Thus, $g(n)\in S\cap 2^n$.

\medskip

Satisfying $\mcal{P}_e $ is similar to  Case 1 for satisfying $\mcal{R}_e$.
\bigskip

We now prove the ``moreover'' part. Assuming $f^*\leq_T D'$ is weakly stable and $S$ is $\Pi^{0,D}_1$.
The condition we use is still $(F,B,f)$ but with the additional assumption that $f$ is of the form
 \begin{align}\label{tt22wklsection3eq0}
   g\otimes g_{0}\otimes \cdots\otimes g_{s}
\text{ (for some $s$) where }g\leq_T D'
\text{ is weakly stable and }
g_{0}\oplus\cdots\oplus g_{s}
\text{ is } D\text{-low},
\end{align}
and an index for verifying  $(g_{0}\oplus\cdots\oplus g_{s})'\leq_T D'$ is incorporated in the tuple.
We modify the construction to ensure  $(D\oplus G)'\leq_T D'$. This is achieved by
 adding a lowness requirement to ``force the jump'':
\begin{itemize}
\item $\mcal{N}_e$: Either $\Psi_e^{D\oplus F_{e+1}}(e)\downarrow$ or
 for all finite perfect  trees $F'\succeq F_{e+1}$
 with $F'\setminus F_{e+1}$  homogeneous for $f_{e+1}$, $\Psi_e^{D\oplus F'}(e)\uparrow$.
\end{itemize}
Note that in Case 1, by Lemma \ref{lem3.9vi}  (ii) (and Lemma \ref{lem3.9x}  (ii) resp.),
the tree $F_{e+1}$ (and the set $B_{e+1}$ resp.) can be computed uniformly in $g_e\oplus (g_{e,0}\oplus\cdots\oplus g_{e, s_e})'$, where $f_e= g_e\otimes g_{e,0}\otimes\cdots\otimes g_{e, s_e}$,
and by Lemma \ref{lem3.9x}  (ii),  $f_{e+1}$ also satisfies
(\ref{tt22wklsection3eq0}).
In Case 2, instead of selecting an arbitrary
$g_m$ from $ Q^n_{V_m}$, we  apply the  Low Basis Theorem to obtain  a $g_{e+1,m}\in Q^n_{V_m}$ so that
 $g_{e,0}\oplus\cdots\oplus g_{e,s_e}\oplus g_{e+1, 0}\oplus\cdots\oplus g_{e+1, r-1}$
is $D$-$low$. Thus $f_{e+1}=f_e\otimes g_{e+1,0}\otimes\cdots\otimes g_{e+1, r-1}$ also satisfies (\ref{tt22wklsection3eq0}) (where now $s=r-1$).
Moreover, whether i Case 1 or Case 2 holds is  also decidable  by $D'$.

Finally, we show how to satisfy the requirement $\mcal{N}_e$
(which is similar to  satisfying $\mcal{R}_e$):
Consider the $\Pi^{0,D}_1$-class
\begin{eqnarray*}
Q_e &=& \big\{f: f \text{\ is a $k$-\treesplit\ on $[B_e]^\preceq$, and for every finite perfect \treeset\ }F\succeq F_e,\\
&&\text{if $F\setminus F_e\subseteq [B_e]^\preceq$ and $F\setminus F_e$ is homogeneous for $f$, then\ } \Psi_e^{D\oplus F}(e)\uparrow \big\}.
\end{eqnarray*}
Depending on whether  $Q_e=\emptyset$, one can repeat the argument in Cases 1 and 2 above to arrive at the  desired extension.
\end{proof}

\begin{corollary} \label{coro103}
$\tcoh$ admits avoidance of $1$-enumeration and bounded enumeration.
Moreover, if $\tt_{k'}^1$ admits strong avoidance of bounded enumeration for all $k'\in\omega$, then so does $\tcoh$.
\end{corollary}

\begin{proof}
The first part, avoidance of $1$-enumeration, follows immediately from Proposition \ref{prop4}, Proposition \ref{prop2} and Theorem \ref{th3}.
Avoidance of bounded enumeration can be proved in the same way as that of $1$-enumeration.
The ``moreover'' part follows from Proposition \ref{prop4}.
\end{proof}

\subsection{$\tcoh$ and other combinatorial principles}
The following two propositions are shown using methods similar to those in the proof of Theorem \ref{th3}.

\begin{proposition} \label{prop5}
$\tcoh$ preserves non-$\Sigma^0_1$~definition, i.e.~for every $D\subset\omega$, every $X$ that is not $\Sigma^0_1$ in $D$,
and every $D$-computable coloring $C: [2^{<\omega}]^2\rightarrow k$, there exists a $\tcoh$ solution $G$ of $C$
such that $X$ is not $\Sigma^0_1$ in  $D\oplus G$.
\end{proposition}

\begin{proof} We give a sketch of the proof.
We build a sequence of tuples $(F_e,B_e,f_e)$ such that
for every $e$ the $e^{th}$ requirement ($\mcal{R}_e$: $\Psi_e^{D\oplus G}\ne X$) is
forced in the following way:
\begin{itemize}
\item  There is a witness $m$ such that either, $\Phi_e^{D\oplus F_e}(m)\downarrow=1 \neq X(m)=0$ or
for all perfect trees $G\succeq F_e$ with $G\setminus F_e$
homogeneous for $f_e$, $ \neg (\Phi_e^{D\oplus G}(m)\downarrow = 1 \wedge X(m) = 1)$.
\end{itemize}
One  forces $\mcal{R}_e$ by applying   strategies similar to those used in the proofs of Proposition \ref{prop4} and Theorem \ref{th3}.
\end{proof}

\begin{proposition} \label{prop3}
$\tcoh$ preserves countable hyperimmunity. i.e.~for any $D\subset\omega$ and $D$-computable coloring $C: [2^{<\omega}]^2\rightarrow k$,
if $\{X_n: n \in \omega\}$ is a collection of sets each  \hyperimmune\ relative to $D$,
then there exists a $\tcoh$ solution $G$ of $C$ such that $X_n$ is \hyperimmune\ relative to $D\oplus G$ for each $n$.
\end{proposition}

\begin{proof} (Sketch)
  Using the method in  Patey \cite{Patey2015Iterativea}, one can turn the preservation of hyperimmunity into  satisfying requirements similar to $\mcal{R}_e$ above.
\begin{itemize}
  \item $\mcal{R}_{e,n}$:
  \text{Either for some }$m$, $\Psi_e^{D\oplus G}(m)\downarrow\subseteq \overline{X}_n$,
  \text{ or }$\Psi_e^{D\oplus G}$\text{ is not total.}
\end{itemize}
  Given $(F_e,B_e,f_e)$, to force $\mcal{R}_{e,n}$,
  for each $m\in\omega$, consider
 \begin{align}\nonumber
  Q_m = \big\{
  f: f\text{ is a }k\text{-\treesplit\ on }[B_e]^\preceq\text{
  and for every finite perfect \treeset\ }F\succeq F_e, \\ \nonumber
  \text{ if }F\setminus F_e\subseteq [B_e]^\preceq
  \text{ and }F\setminus F_e\text{ is homogeneous for }
  f\text{ then }\Psi_e^{D\oplus F}(m)\uparrow
  \big\}.
  \end{align}
  If $Q_m\ne\emptyset$ for some $m$, then select an $\h{f}\in Q_m$
  and $(F_e,B_e,f_e\otimes \h{f})$ forces $\mcal{R}_{e,m}$
  (by fulfilling the $\Pi_1^0$-clause).
  Suppose  $Q_m=\emptyset$ for all $m\in\omega$. Since $X_n$ is hyperimmune,
   there exist an  $m^*$ and a finite set $\mcal{F}$
   of finite perfect trees witnessing $Q_{m^*}=\emptyset$ (by compactness argument)
   such that $\Psi_{e}^{D\oplus F}(m^*)\downarrow\subseteq \overline{X}_n$
   for all $F\in\mcal{F}$.
   Suppose $F^*\in\mcal{F}$ is homogeneous for some \conciseversion\ of
   $f_e$. Then  by Lemma \ref{lem3.9x} (i), for some $B^*$ and $k$-\treesplit\ $\h{f}$,
   $(F^*,B^*,\h{f})$
   is a condition extending $(F_e,B_e,f_e)$.
\end{proof}

\begin{corollary} \label{CTTandWWKL}
Over $\msf{RCA}$, $\tcoh$ does not imply either $\msf{WWKL}$ or $\msf{SRT}_2^2$.
\end{corollary}

\begin{proof}
Let $S$ be a pruned $\Pi^0_1$-definable tree in which every path through $S$ is $1$-random.
Then $S$ does not admit a computable $1$-enumeration and so by Corollary \ref{coro103}
there is a model $\mathfrak{M}$ of $\tcoh$ for which no $X\in \mathfrak{M}$ computes a $1$-enumeration of $S$.
However, every model of $\msf{WWKL}_0$ contains a second order member $X$ that is a path in $S$,
and such an $X$ clearly computes a $1$-enumeration of $S$. It follows that  $\tcoh$ does not imply $\msf{WWKL}_0$.

For the ``$\msf{SRT}_2^2$" part, let $C:\omega\rightarrow 2$ be a $\Delta_2^0$-coloring such that $C^{-1}(i)$ is hyperimmune for $i< 2$.
Let $C$ be induced by the computable stable coloring $\t{C}:[\omega]^2\rightarrow 2$.
By Proposition \ref{prop3}, there exists a model $\mathfrak{M}$ of $\tcoh$ such that relative to every member $G\in \mathfrak{M}$,
$C^{-1}(i)$ is hyperimmune for $i<2$.  On the other hand, any model containing a member $\t{G}$ which is a solution of $\t{C}$ must satisfy the fact
that relative to $\t{G}$ at least one of $C^{-1}(i)$ is not \hyperimmune.
Thus $\tcoh$ does not imply $\msf{SRT}_2^2$.
\end{proof}



We end this section with a result relating $\msf{CTT}^2_k$ to the principle $\msf{DNR}$.
Recall that $\mathsf{DNR}$ states: for every partial $\omega$-valued function $g$,
there is a total function $h$ such that for all $e$, if $\Phi_e^g(e)\downarrow$ then $\Phi^g_e(e)\ne h(e)$.
Hirschfeldt, Jockusch, Kjos-Hanssen, Lempp and Slaman \cite{Hirschfeldt2008strength} proved that over $\mathsf{RCA}_0$,
$\mathsf{SRT}^2_2$ implies $\mathsf{DNR}$ whereas the cohesive set principle $\mathsf{COH}$ does not.

\begin{proposition}
$\tcoh$ implies $\msf{DNR}$.  Consequently, $\mathsf{COH}$ does not imply $\tcoh$.
\end{proposition}

\begin{proof}
We first show that every $\omega$-model of $\tcoh$ contains a DNR function, i.e.~a function that diagonalizes against every $\varphi_e(e)$ whenever the latter is defined.
As the proof relativizes to any real, it  implies that every $\omega$-model of $\tcoh$ satisfies $\msf{DNR}$.

Let $\{\varphi_e(x): e\in\omega\}$ be an effective list of all unary partial recursive functions.
For any infinite perfect tree $G$, let $h^G(e)$ be the G\"odel number of its initial segment $G_e\cong  2^{<(e+1)}$.
Clearly, $h^G\le_T G$  and is total.
We define a weakly stable $2$-\treesplit\ $f\le_T\emptyset'$ on $2^{<\omega}$  such that for any $G$ which is homogeneous for $f$,
$h^G(e)\neq \varphi_e(e)$ whenever $\varphi_e(e)\downarrow$.  By Proposition \ref{propcjsctt}, a $G$ homogeneous for $f$ always exists in a model of $\tcoh$ and hence $h^G$ is the required DNR-function.

At stage $s=0$,  define $f(\varepsilon)=\{0,1\}$.

At stage $s=e+1$, we first diagonalize against $\varphi_e(e)$.  Use $\emptyset'$ to see if $\varphi_e(e)\downarrow$, if not, do nothing.
Otherwise,   see if it is a G\"odel number of a tree $F$  isomorphic to $2^{<(e+1)}$. If not, do nothing.  Otherwise,
choose the least triple $(\sigma, \rho_0, \rho_1) \in F^3$ such that (1) $|\sigma|\geq e$,  $\rho_0$,  $\rho_1$ are incompatible and each extends $\sigma$,
and  (2)  $|\sigma|\neq |\sigma'|$
 for any triple $(\sigma', \rho'_0, \rho'_1)$ that was  used at an earlier stage.
Such a triple exists because $F\cong 2^{<(e+1)}$ and at each stage  at most one triple was used.
 For $i=0,1$ and $\rho\succeq \rho_i$, we will make sure that:
   \begin{align}
   f(\rho)(|\sigma|)=\{i\} \ \ \ \ (\mcal{R}_{\sigma,\rho_0,\rho_1}).
   \end{align}
Declare the triple $(\sigma, \rho_0, \rho_1)$ as {\em used}.
The requirement $\mcal{R}_{\sigma,\rho_0,\rho_1}$ ensures that $F$ is not  an initial segment of a homogenous tree.
At stage $e+1$, suppose $f$ has been defined up to level $m_e$
(where $m_e$ is large enough say $m_e = \max_{t\leq r_e}\{|\rho^t_0|,|\rho^t_1|\}$
and where $\{(\sigma^t,\rho^t_0,\rho^t_1)\}_{t\leq r_e}$ are triples used before stage $e$)
so that it satisfies  all requirements
appearing before stage $e$ and is stable up to level $e$ (i.e., $|f(\rho)\uhr e| = 1$ for all $\rho$ with $|\rho|= m_e$).
Clearly we can extend $f$   to level $m_{e+1}$ to satisfy the newly appeared requirement at stage $e+1$,
since $\sigma$ has not been used before. Moreover, we can also ensure that $f$ is stable up to $e+1$.
\end{proof}

\begin{remark}
{\rm  The above construction may be carried out in any model of $\msf{RCA}_0$. In particular, for each $e_0$, the set $\{e: \varphi_e(e)\downarrow\}$ is a $\Sigma^0_1$-definable set and hence $\mathfrak{M}$-finite for any $\mathfrak{M}\models \mathsf{RCA}_0$.}
\end{remark}


\section{Separating $\msf{WWKL}$ from $\tt_k^2$} \label{tt22vswklsec4}

The question whether $\tt_2^2$ implies $\msf{WKL}_0$, a natural analog  of the original question for $\msf{RT}^2_2$ solved in Liu\cite{Liu2012RT22}, was raised in \cite{DzhafarovColoring}.
The main result of this paper answers the question negatively by exhibiting a model of $\tt_2^2$ that admits avoidance of bounded enumeration.
Thus an argument similar to that in the proof of Corollary \ref{CTTandWWKL}  works for $\tt^2_2$.
Recall that Dzhafarov, Hirst and Lankins \cite{Dzhafarov2009polarized} (Proposition \ref{DHLdecomp})  showed that $\msf{TT}_k^2 \leftrightarrow \msf{STT}_k^2+ \tcoh$,
and Corollary \ref{coro103} says that $\tcoh$ admits avoidance of bounded enumeration.
It remains to  show that $\stt_k^2$ admits avoidance of bounded enumeration.
Since every computable stable coloring $C: [2^{<\omega}]^2\rightarrow k$ of $\stt_k^2$ naturally induces a $\Delta_2^0$-instance $\t{C}$ of $\tt_k^1$
such that any  solution of $\t{C}$ computes a solution of $C$,
the problem is further reduced to showing that $\tt_k^1$ admits strong avoidance of bounded enumeration.
In this section, we prove that this is indeed the case.

\begin{theorem}\label{th101}
$\tt_k^1$ admits strong avoidance of bounded enumeration.
\end{theorem}

We  begin by  introducing the following notion:

\begin{definition} \label{defavoidancepi01}
A problem $\msf{P}$ admits $\Pi_1^0$-\emph{class avoidance of bounded enumeration} if for any $D\subseteq\omega$ and $S\subseteq 2^{<\omega}$ that
does not admit a $D$-computable bounded enumeration, and any nonempty $\Pi_1^{0,D}$-class $Q$ of $\msf{P}$-instances,
there exists an instance $X\in Q$, a $\msf{P}$-solution $Y$ of $X$ such that $D\oplus Y$ does not compute a bounded enumeration of $S$.
\end{definition}

In the current setting, $Q$ is a $\Pi^{0,D}_1$-class of $k$-colorings on $2^{<\omega}$. We prove Theorem \ref{th101} in two steps: First,
in Theorem \ref{th8} we reduce  strong avoidance of bounded enumeration for $\tt_k^1$ to $\Pi_1^0$-class avoidance of bounded enumeration for $\tt^1_k$.
We then prove the latter in Theorem \ref{th100}.

\begin{restatable}{theorem}{reduce} \label{th8}
$\tt_k^1$ admits strong avoidance of bounded enumeration if and only if it admits $\Pi_1^0$-class avoidance of bounded enumeration.
\end{restatable}

\begin{restatable}{theorem}{classavoidance} \label{th100}
$\tt_k^1$ admits $\Pi_1^0$-class avoidance of bounded enumeration.
\end{restatable}

 Theorem \ref{th101}  allows one to remove the hypothesis in  Corollary \ref{coro103}:
\begin{corollary} \label{strong}
$\tcoh$ admits strong avoidance of bounded enumeration.
\end{corollary}

We prove Theorem \ref{th8} in the next subsection, and   Theorem \ref{th100} in \S \ref{secpi10}.

\subsection{Reduction to $\Pi_1^0$-class avoidance} \label{secreducing}

\reduce*


Strong avoidance obviously implies $\Pi_1^0$-class avoidance. We will prove  the converse.
Since the argument relativizes to  any set $D$ in Definitions \ref{boundenumeration} and \ref{defavoidance}, we may take $D=\emptyset$.
Suppose $S\subseteq 2^{<\omega}$  does not admit a computable bounded enumeration.
Let $C:2^{<\omega}\rightarrow k$ be a $\tt_k^1$ instance.
We show that there exists a  solution $G$ of $C$ such that $G$ does not compute a bounded enumeration of $S$
(under the assumption that $\tt_k^1$ admits $\Pi_1^0$-class avoidance of bounded enumeration).

In the proof we assume that for any Turing functional $\Psi_e$ there exists an $l$, depending on $e$, such that
for any oracle set $X$, $\Psi_e^X$ outputs an $l$-enumeration of $2^{<\omega}$, and for any $X$ and $n$, $\Psi_e^X(n)\downarrow\rightarrow \Psi_e^X(n)\subseteq 2^n$.

Since we are working in $\omega$-models, we may assume that all colors are dense. In fact, we will assume something stronger:
For every $k'\in k$ and any $T\subseteq 2^{\omega}$, if $T$ does not compute a bounded enumeration of $S$,
then $C^{-1}(k')\cap T$ is dense in $T$.  This  is possible   since otherwise one can use the non-density of the color $k'$ in $T$ and work inside a subtree of $T$ which is recursive in $T$ and has no node with color $k'$.
An easy induction on the number of colors  finishes the proof.
This observation enables us to concentrate on achieving avoidance without worrying about the density of colors.
It also circumvents  issues like $S$ being coded in color $0$ (as $C^{-1}(0)$ may otherwise  not be dense in certain trees).

The requirements on avoidance of bounded enumeration are as follows:
\begin{itemize}
\item $\mcal{R}_e$: Either for some $n$, $\Psi_e^G(n)\downarrow$ and $\Psi_e^G(n) \cap S=\emptyset$ or $\Psi_e^G$ is not total.
\end{itemize}

We introduce  a  method  that is derived from  Mathias forcing.
As in Section \ref{tt2wklsec3}, let $B$  denote a finite set of pairwise incompatible nodes on $2^{<\omega}$,
and let $X$ and $Y$  denote {\em forests with cone base $B$}. These are sets  of the form $\bigcup \{U_{\sigma}:\sigma\in B\}$
where each $U_{\sigma}$ is an infinite perfect tree.

A \emph{condition} is a pair $(F,X)$ such that $F$ is a finite perfect tree and $X$ is a forest with cone bases $\leaf(F)$ such that
$X$ does not compute a bounded enumeration of $S$. Consequently no $U_{\sigma}\subseteq X$ computes a bounded enumeration of $S$, because $B$ is finite.
Moreover, by the remark above, every color is dense in $X$.
A condition $(F_1, X_1)$ \emph{extends} a condition $(F_0,X_0)$, written $(F_1,X_1)\leq (F_0,X_0)$,
if $F_1\succeq F_0$ and $F_1\cup X_1\subseteq F_0\cup X_0$.

We say that a condition $(F,X)$ \emph{satisfies} $\mcal{R}_e$ on color $i$ if $F\subseteq C^{-1}(i)$, and there is an $n$ such that $\Psi^F_e(n)\downarrow$ and $\Psi^F_e(n)\cap S=\emptyset$,
or for all $F'\succeq F$ with $(F'\setminus F)\subseteq X$, $\Psi^{F'}_e(n)\uparrow$.

Failing to satisfy an $R_e$ may offer us a $G\leq_T C$ which computes a bounded enumeration of $S$.
While ordinarily this would immediately lead to a contradiction,
it is not so when
 dealing with strong avoidance.   A new strategy is required  to handle this situation.
We  begin with identifying the conditions which cannot be  extended to satisfy an   $R_e$.

\begin{definition} \label{encode}
Let $i<k$. We say that a condition $(F,X)$ is {\em bad for $\mathcal{R}_e$ on color $i$} if $F\subseteq C^{-1}(i)$ and no  extension $(\t{F},\t{X})$
with $\t{F}\subseteq C^{-1}(i)$ satisfies $R_e$.
\end{definition}

The following steps enables one to select a ``good'' color:  For $i=0$, check if for all infinite perfect $T\subseteq 2^{<\omega}$,
such that $T$ does not compute a bounded enumeration of $S$,
there exist a condition $(F, X)$ with $(F\cup X)\subseteq T$, an index $e$ such that $(F,X)$ is bad for $\mathcal{R}_e$ on color $0$.
If the answer is ``no'', then  there is an infinite perfect $T\subseteq 2^{<\omega}$ which does not compute any bounded enumeration of $S$
such that for each condition $(F, X)$ with $(F\cup X)\subseteq T$,
for each index $e$, there is an extension $(\t{F},\t{X})\leq (F,X)$ satisfying $\mathcal{R}_e$ on color $0$.
We take this $T$ and the ``good'' color $0$, and stop the process.
If the answer is ``yes'',  repeat the steps above for $i=1$.  This process continues sequentially for  colors in $k$ and
 yields two possible outcomes: either
 \begin{enumerate}
 \item [(I)]    There exist an infinite perfect tree $T$  not computing a bounded enumeration of $S$, and a color $i<k$ such that for each condition $(F, X)$ with $(F\cup X)\subseteq T$,
for each index $e$, there is an  $(\t{F},\t{X})\leq (F,X)$ satisfying $\mathcal{R}_e$ on color $i$, or
\item [(II)]  For any color $i<k$,
and any infinite perfect tree $T$ not computing a bounded enumeration of $S$, there exist a condition $(F, X)$ with $(F\cup X)\subseteq T$, an index $e$ such that $(F,X)$ is bad for $\mathcal{R}_e$ on color $i$.
\end{enumerate}

For Case  (I), we can easily build an solution $G$ with color $i$:

\begin{lemma} \label{lem18}
Let $T$ be an infinite perfect tree such that $T$ does not compute a bounded enumeration of $S$. Assume that  for each condition $(F, X)$ with $(F\cup X)\subseteq T$
and index $e$, there is an  $(\t{F},\t{X})\leq (F,X)$ satisfying $\mathcal{R}_e$ on color $i$.
Then there exists an infinite perfect \treeset\ $G\subseteq T \cap C^{-1}(i)$ that does not compute a bounded enumeration of $S$.
\end{lemma}

\begin{proof} We build a sequence of conditions $(F_e,X_e)_{e\in \omega}$ as follows.
Let $(F_0,X_0) $ be the condition $ (\emptyset, T)$.  Suppose  $(F_e,X_e)$ is defined  such that  $F_e\subseteq C^{-1}(i)$ and
$(F_e, X_e)$ satisfied $R_{e-1}$ on color $i$.
By assumption, there exists an $(F, X)$ extending $(F_e, X_e)$ which satisfies $R_{e}$ on color $i$.
Let  $(F_{e+1}, X_{e+1})$ be the least such $(F, X)$.  Let $G=\bigcup \{F_e:e\in \omega\}$.
The construction guarantees that $G$ satisfies all requirements $R_e$.
Moreover $G$ is infinite by an argument similar to that in the proof of Proposition \ref{prop4}.
\end{proof}

We will derive a contradiction for Case II which   will occupy  the rest of this subsection.
Recall in this case that
  for any color $i<k$ and any infinite perfect tree $T$ not computing a  bounded enumeration of $S$,
there exist  $(F, X)$ with $(F\cup X)\subseteq T$ and
 $e$ such that $(F,X)$ is bad for $\mathcal{R}_e$ on color $i$.

 We first select $k$ families of finite perfect trees $\mcal{F}_0, \dots, \mcal{F}_{k-1}$ as follows:
Let $B_{-1}=\{\varepsilon_0\}$ where $\varepsilon_0$ is the root of the tree $2^{<\omega}$ and $X_{-1}=2^{<\omega}$.
To define  $\mcal{F}_0$, apply Case (II) to the color $0$ and  $T=2^{<\omega}$ to obtain   $(F_0, X_0)$ and $e(F_0)$ such that
$(F_0,X_0)$ is bad for $\mathcal{R}_{e(F_0)}$ on the color $0$.  Let $\mcal{F}_0=\{F_0\}$ and $B_{0}=\leaf(F_0)$.

Inductively, suppose that $\mcal{F}_{i}=\{F_i^{\sigma}: \sigma\in B_{i-1}\}$, $B_i=\bigcup \{\leaf(F): F\in \mcal{F}_i\}$
and $X_i=\bigcup \{U_{\sigma}: \sigma\in B_i\}$ are defined.
For each $\sigma\in B_i$, observe that $U_{\sigma}$ does not compute a bounded enumeration of $S$.
Apply the case assumption to the color $i+1$ and the tree $U_{\sigma}$ to conclude that
there exist  $(F_{i+1}^{\sigma}, X_{i+1}^{\sigma})$ with $F_{i+1}^{\sigma}\cup X_{i+1}^{\sigma}\subseteq U_{\sigma}$ and
 $e(F_{i+1}^{\sigma})$ such that $(F_{i+1}^{\sigma}, X_{i+1}^{\sigma})$ is bad for $\mathcal{R}_{e(F_{i+1}^{\sigma})}$ on the color $i+1$.
Let $\mcal{F}_{i+1} = \{F^\sigma_{i+1}:\sigma\in B_i\}$, $B_{i+1}=\bigcup \{\leaf(F): F\in \mcal{F}_{i+1}\}$
and $X_{i+1}=\bigcup \{X_{i+1}^{\sigma}: \sigma\in B_i\}$ which can clearly  be written as $\bigcup \{U_{\sigma}: \sigma\in B_{i+1}\}$.

We refer to the collection $\{\mcal{F}_0, \dots, \mcal{F}_{k-1}\}$ as a {\em $k$-\hierarchy}.
Informally,  a $k$-\hierarchy\ is a family of $k$-layers of finite perfect trees, such that above  the   leaves of  trees in one layer lie the trees in the next layer.
Let $B=\bigcup \{\leaf(F): F\in \mcal{F}_{k-1}\}$ be the set of leaves of the trees in the last layer and let $X=X_{k-1}$, the  forest with cone basis $B$.
Notice that for any tree $F$ in any $\mcal{F}_i$, every node $\sigma$ of $F$ is extended by some $\tau$ in $B$.
We  refer to $B$ as  a ``{\em cover} of $F$''.

The main property of a $k$-\hierarchy\ that we will exploit  is in  the following lemma
(note that the number $k$ in the $k$-\hierarchy\  is  the number of colors).

\begin{lemma}\label{lem12}
Let $\mcal{F}_0,\dots,\mcal{F}_{k-1}$ and $B$ be as above. Then for any $E$ which is a cover of $B$,
and for any function $g: E\to k$, there exist $i\in k$, $F\in\mcal{F}_{i}$ and a set $\t{E} \subseteq E$ which covers $F$ and
$\t{E}\subseteq g^{-1}(i)$.
\end{lemma}

\begin{proof}
We prove the lemma by induction on $k$.  For $k = 1$, the conclusion clearly holds.
Suppose the conclusion holds for any $k-1$-hierarchy.  Let $F_0\in\mcal{F}_0$.
If for every $\sigma\in \leaf(F_0)$ there exists a $\tau\in E$ such that $\sigma\preceq \tau$ and $g(\tau) = 0$,
then we are done by taking $i=0$, $F=F_0$ and $\t{E}=\{\tau\in E: \tau$ extends some leaf in $F_0$ and $g(\tau) = 0\}$.

Suppose for some $\sigma\in \leaf(F_0)$, we have $g(\tau)\neq 0$ for every $\tau\in E$ and $\tau\succeq \sigma$.
Consider the $(k-1)$-\hierarchy\ obtained by restricting the original $k$-\hierarchy\ to the  cone $[\sigma]^\preceq$,
and let $E'$ be the restriction of $E$ to the same cone. Now $g$ restricted to $E'$ has range $\subseteq \{1,\dots,k-1\}$.
The conclusion follows by the inductive hypothesis.
\end{proof}

Let $d$ and $l$ be positive integers.  By an {\em $l$-partition of $d$}, we mean a partition $\{W_0,\cdots, W_{l-1}\}$ of $\{0,\dots,d-1\}$ (note that we do not require the sets in a partition to be nonempty).
Recall the notion of a blocking set (or hitting set) studied in combinatorial mathematics and computer science:
Given a family of finite sets $\{V_m\}_{m< d}$ and a finite set $U$, we say that $U$ is a {\em blocking set of} $\{V_m\}_{m< d}$
if for all $m<d$, $V_m\cap U\neq \emptyset$.

\begin{definition} \label{defdisperse}
Let $d, l>0$ and $\{V_m\}_{m< d}$ be a family of nonempty  finite sets.
We say that $\{V_m\}_{m< d}$ is $l$-\emph{\disperse} if for every $l$-partition $\{W_0,\dots, W_{l-1}\}$ of $d$, there exists an $i<l$ such that $W_i\ne\emptyset$ and $\bigcap_{m\in W_{i}} V_m = \emptyset$.

We say that $\{V_m\}_{m< d}$ is $(k,l)$-\emph{\disperse} if for every $k$-partition of $d$,
say $W_0,\cdots,W_{k-1}$, there exists a $k'<k$ such that $\{V_m\}_{m\in W_{k'}}$ is $l$-\disperse.
\end{definition}

\begin{lemma} \label{tt22vswkllem17}
\
\begin{enumerate}
\item [{\rm (i)}] $\{V_m\}_{m< d}$ is $l$-\disperse \ if and only if it has no blocking set of size $\leq l$,
i.e., for any $U$ with $|U|\leq l$, $U\cap V_{m'}=\emptyset$ for some $m'<d$.
\item [{\rm (ii)}] $\{V_m\}_{m< d}$ is $k\cdot l$-scattered if and only if it is $(k, l)$-\disperse.
\end{enumerate}
\end{lemma}

\begin{proof}
To prove (i), first assume that $\{V_m\}_{m<d}$ is $l$-scattered.  Suppose for the sake of contradiction that for some least $l'\le l$ and $U = \{\rho_0,\dots,\rho_{l'-1}\}$,
$U\cap V_m\ne\emptyset$ for all $m< d$.  Consider the $l'$-\partition\ $\{W_0,\dots, W_{l'-1}\}$ of $d$ where $W_{i} = \big\{m: \rho_{i}\in V_m\big\}\setminus (W_0\cup\dots\cup W_{i-1})$.
Now if $W_i\ne\emptyset$ then $\bigcap_{m\in W_{i}} V_m\ne\emptyset$ since it contains $\rho_{i}$.
This contradicts the $l$-scattering property of $\{V_m\}_{m<d}$.

Conversely, suppose $\{V_m\}_{m<d}$ is not $l$-scattered. Let $W_0,\dots, W_{l-1}$ be an $l$-partition of $d$ that witnesses this.
This means that $\bigcap_{m\in W_{i}} V_m\ne\emptyset$ for all $i<l$.
Let $U=\{\rho_{i}: \rho_{i}\in \bigcap_{m\in W_{i}} V_m\}$.  Then $|U|\le l$ and $U\cap V_m\ne\emptyset$ for each $m<d$.

For (ii), assume that $\{V_m\}_{m< d}$ is $k\cdot l$-scattered but not $(k,l)$-scattered.
Let $W_0,\dots,W_{k-1}$ be a $k$-partition of $\{d\}$ witnessing this failure, i.e.~for every $k'\in k$,
$\{V_m\}_{m\in W_{k'}}$ is not $l$-\disperse.  By (i), this means that for each $k'\in k$,
there exists a blocking set $U_{k'}$ of size $\leq l$.  Let $U=\bigcup_{k'<k}U_{k'}$.
Then $|U|\leq k\cdot l$ which is a blocking set of $\{V_m\}_{m< r}$, a contradiction.
The proof of the converse is immediate and left to the reader.
\end{proof}

We now return  to  deriving a contradiction under the assumption of Case (II). Here is the current status:  There is a $k$-\hierarchy\ $\mcal{F}_0,\dots,\mcal{F}_{k-1}$, a  cover  $B$ and a forest
$X$  with cone basis $B$,  satisfying the following:   $X$ does not compute  a bounded enumeration of $S$, and  for any $F\in \mcal{F}_i$ there exist $Y=Y(F)\subseteq X$ and $e(F)$ such that
  $(F,Y)$ is bad for $R_{e(F)}$ on the color $i$.

For each $n\in\omega$ and $V\subset 2^{<n}$, consider the following $\Pi_1^{0,X}$-class of $k$-colorings $\h{C}$ on $2^{<\omega}$:
\begin{align}
Q^n_V =\big\{\h{C}:\ \ &\text{For all }i\in k \text{ and } F\in \mcal{F}_i\text{ and for all finite perfect \treeset s\ } \t{F}\succ F \text{ in color } i
\nonumber,\\
&\text{ if } \t{F}\subseteq (F\cup Y(F)) \text{ and } \Psi^{\t{F}}_{e(F)}(n)\downarrow\text{ then }
\Psi^{\t{F}}_{e(F)}(n)\cap V\ne\emptyset\ \big\}. \nonumber
\end{align}
Suppose $|B| =u$.
We first make a claim.

\begin{claim} \label{claim2}
$Q^n_{S\cap 2^{<n}}\ne\emptyset$ for all $n\in\omega$.
\end{claim}

\noindent  {\it Proof of Claim.}
Suppose otherwise and let $n\in\omega$ be such that $Q^n_{S\cap 2^{<n}}=\emptyset$. Then in particular the given coloring $C$ is not in $Q^n_{S\cap 2^n}$.
Thus there exist $i\in k, F\in \mcal{F}_i$,  $\t{F}\succeq F$ and an index $e=e(F)$ such that
$\t{F}\subseteq (F\cup Y(F))$, $\t{F}\subseteq C^{-1}(i)$, $\Psi^{\t{F}}_{e}(n)\downarrow$ and $\Psi^{\t{F}}_{e}(n)\cap S\cap 2^n=\emptyset$.
However, this contradicts the assumption that $e(F)$ witnesses $(F,Y(F))$ being bad on the color $i$ of $C$.
\qed

\bigskip

Let $l =\max_{F\in \mcal{F}_i,i\in k}\{l(F):\Psi^F_{e(F)} \text{ is an\ } l(F)\text{-enumeration}\}$.
We divide Case II  into two subcases, similar to Case 2 and Case 3 in Theorem \ref{th3}.
\bigskip

\noindent\textbf{Case 1.} For some $n, d\in\omega$, there exists a $(k^u, l)$-\disperse\ collection $\{V_m\}_{m\leq d}$ of subsets of $2^n$ such that
$Q^n_{V_m}\ne\emptyset$ for each $m\leq d$.

\begin{claim} \label{claimCase1}
Suppose $\tt_k^1$ admits $\Pi_1^0$-class avoidance of  bounded enumeration.
There exist a cover $E\subseteq X$ of $B$, a forest $Z \subseteq X$ with cone basis $E$,
a set $D\geq_T Z$ not computing a bounded enumeration of $S$, and for each $m<d$, there exist a $\Pi^{0,D}_1$-class $Q^n_{V_m,d-1}\subseteq Q^n_{V_m}$  and
a coloring $\h{C}_{m}\in Q^n_{V_m,d-1}$, such that for all $\tau\in E$,
there is a color $i_{m,\tau}\in k$ with $Z\cap [\tau]^{\preceq} \subseteq \h{C}_m^{-1}(i_{m,\tau})$.
\end{claim}

\noindent {\it Proof of Claim.} The proof is similar to the ``moreover'' part of Proposition \ref{prop4}.

List all members of $B$ from left to right as $\sigma_0,\dots,\sigma_{u-1}$ and let $U_v=X\cap [\sigma_v]^{\preceq}$.
Then $X=\bigcup\{U_v: v<u\}$.
We define $\hat{C}_m$ by induction on $m<d$ and for each $m$, we define $i_{m,v}$ and a perfect tree $Y_{m,v}$ with root $\tau_{m,v}$ by induction on $v< u$.

For $m = 0$ and $v=0$: Notice that $U_{0}\leq_T X$ does not compute a bounded enumeration of $S$.
Consider $Q^n_{V_0}\uhr U_{0}=\{\h{C}\uhr U_{0}: \h{C}\in Q^n_{V_0}\}$ which is a $\Pi^{0,X}_1$-class.
By the property of $\Pi_1^0$-class avoidance of bounded enumeration, there is     an instance $\h{C}_{0, 0}\in Q^n_{V_0}$
(strictly speaking, $\h{C}_{0,0}$ is an extension of some element in $Q^n_{V_0}\uhr U_{0}$),
and a solution $Y_{0, 0}$ which is an infinite perfect  subtree  of $U_{0}$, and an $i_{0,0}\in k$ such that $Y_{0,0}\subseteq \h{C}_{0,0}^{-1}(i_{0,0})$.
Then  $D_{0,0}=X\oplus Y_{0,0}$ does not compute a bounded enumeration of $S$. Let $\tau_{0,0}$ be the root of $Y_{0,0}$.
Let $Q^n_{V_0, 0} = \{\h{C}\in Q^n_{V_0}: Y_{0,0}\subseteq \h{C}^{-1}(i_{0,0})\}$.
This is a nonempty $\Pi_1^{0,D_{0,0}}$-class of colorings on $2^{<\omega}$ since it contains $\h{C}_{0,0}$.
For $m = 0$ and $v=1$, replace $Q^n_{V_0}$ by $Q^n_{V_0,0}$ and $U_0$ by $U_1$ respectively, and repeat the above procedure,
yielding an infinite perfect tree $Y_{0,1}\subseteq  U_1$, a coloring $\h{C}_{0, 1}\in Q^n_{V_0,0}$ and an $i_{0,1}\in k$ such that $Y_{0,1}\subseteq U_{1}\cap \h{C}_{0,1}^{-1}(i_{0,1})$
and $D_{0,1}=D_{0,0}\oplus Y_{0,1}$ does not compute a bounded enumeration of $S$.
Note that $C_{0,1}\in Q^n_{V_0,0}$ automatically implies $Y_{0,0}\subseteq C_{0,1}^{-1}(i_{0,0})$.
Let $Q^n_{V_0, 1} = \{\h{C}\in Q^^n_{V_0,0}: Y_{0,1}\subseteq \h{C}^{-1}(i_{0,1})\}$ and let $\tau_{0,1}$ be the root of $Y_{0,1}$.
Again this is a nonempty $\Pi_1^{0,D_{0,1}}$-class of colorings on $2^{<\omega}$ since $C_{0,1}\in Q^n_{V_0,1}$.
Therefore, after $v$ traverses through $0,\dots,u-1$, we obtain a set $E_0=\{\tau_{0,v}: v<u\}$ which is a cover of $B$,
a $Y_0 = \bigcup \{Y_{0,v}: v<u\}$ which is a forest with cone basis $E_0$, a set $D_0=X\oplus Y_0$ which does not compute a bounded enumeration of $S$,
 a coloring $\hat{C}_0\in Q^n_{V_0}$ and a color $i_{0,v}$ for each $v< u$ such that $Y_{0,v} \subseteq \hat{C}_{0}^{-1}(i_{0,v})$ for all $v<u$.

Now for $m = 1$, we repeat the above steps with $X$ replaced by $Y_0$, and obtain
\begin{enumerate}

\item A cover $E_1=\{\tau_{1,v}:v<u\}$  of $E_0$;
\item A forrest $Y_1 = \{Y_{1,v}: v<u\}\subseteq Y_0$  with cone basis $E_1$;
\item   $D_1=D_0\oplus Y_1$ which does not compute a bounded enumeration of $S$;
\item A   coloribg $\hat{C}_1\in Q^n_{V_1}$ and
\item A color $i_{1,v}$ for each $v< u$ such that $Y_{1,v} \subseteq \hat{C}_1^{-1}(i_{1,v})$ for all $v< u$.
\end{enumerate}

Inductively repeating these steps for $m\le d-1$, we arrive at a $D =D_{d-1}$ not computing a bounded enumeration of $S$,
a cover $E=E_{d-1}$  of $B$ and  a forest $Z = Y_{d-1}\subseteq X$  with cone basis $E$,
together with a coloring $\hat{C}_m\in Q^n_{V_m}$. for each $m<d$.  and a color $i_{m,v}$ for $v<u$, such that $Y_{m,v}\subseteq \hat{C}_m^{-1}(i_{m,v})$.
\qed

\bigskip

Define a family of subsets $\{W_g: g\in k^u\}$ of $\{0,\cdots,d-1\}$ as follows:
For each $m<d$, put $m\in W_g$ if the function $v\mapsto i_{m,v}$ is $g$ (recall that members in $k^u$ are functions from $u$ to $k$).
  Then $\{W_g: g\in k^u\}$ is a $k^u$-partition of $d$.
 By the $(k^u,l)$-scattering  of $\{V_0,\dots,V_{d-1}\}$, there exists a $g^*$ such that$\{V_m: m\in W_{g^*}\}$ is $l$-\disperse.

Applying  Lemma \ref{lem12} to $E$ and $g^*$, which may be considered to be a map from $E$ to $k$,
we  conclude that there exist  a $k^*\in k$, an $F^*\in \mcal{F}_{k^*}$ and an  $E^*\subseteq E$ such that $E^*$ covers $F^*$ and $E^*\subseteq (g^{*})^{-1}(k^*)$.
By the definition of $W_{g^*}$, we have that for all $m\in W_{g^*}$ and $v<u$, if $\tau_{d-1,v}\in E^*$ then $i_{m,v} = k^*$.
Hence, by the definition of $i_{m,v}$,
$Y_{d-1,v}\subseteq Y_{m,v}\subseteq \h{C}_m^{-1}(k^*)$.  Let $Y^*=\bigcup \{Y_{d-1,v}: \tau_{d-1,v}\in E^*\}$. Note that $(F^*,Y^*)$ is a condition.
By the definition of $Q^n_{V_m}$ and since $\h{C}_m\in Q^n_{V_m}$, we have:

\begin{enumerate}

 \item [(*)] For any $m\in W_{g^*}$ and  $F'\succeq F^*$ with $F'\setminus F^* \subseteq Y^*$,
if $\Psi_{e(F^*)}^{F'}(n)\downarrow$ then $\Psi_{e(F^*)}^{F'}(n)\cap V_m\ne\emptyset$.
\end{enumerate}
.
Suppose that there is an  $F'\succeq F^*$ with $F'\setminus F^* \subseteq Y^*$ and $\Psi_{e(F^*)}^{F'}(n)\downarrow$.
 Then $|\Psi_{e(F^*)}^{F'}(n)|\leq l$ by the definition of $l$, and so $\{W^*_{l'}\}_{l'\le l}$, where
$W^*_{l'}=\{m: m\in W_{g^*}\wedge |\Psi_{e(F^*)}^{F'}(n)|=l'\}$,  is an $l$-partition of $W_{g^*}$. Furthermore, there is a unique $l^*\le l$ such that
$W^*_{l^*}\ne \emptyset$, and in which case $W^*_{l^*}=W_{g^*}$.
Then by the $l$-scattering of  $\{V_m\}_{m\in W_{g^*}}$,  $\bigcap_{m\in W_{g^*}} V_m=\emptyset$.
Therefore, by Lemma \ref{tt22vswkllem17} (i),  $\Psi_{e(F^*)}^{F'}(n)\uparrow$
 for any $F'\succeq F^*$ with $F'\setminus F^* \subseteq Y^*$.
Thus $(F^*,Y^*)$ satisfies $\mcal{R}_{e(F^*)}$.  However $(F^*,Y^*)\leq (F^*,Y)$ and $(F^*, Y)$ is bad for $\mcal{R}_{e(F^*)}$ by assumption, which is a
 contradiction.

\medskip

\noindent\textbf{Case 2.} Otherwise, i.e.~for any $n\in\omega$: $\big\{V\subseteq 2^n: Q^n_V\ne\emptyset \big\}$ is not $(k^u,l)$-\disperse.

By Lemma \ref{tt22vswkllem17} (ii),  $\{V\subseteq 2^n: Q^n_V\ne\emptyset \}$ is not $k^u\cdot l$-\disperse \ for each $n$.
Therefore there is a blocking set $U_n\subseteq 2^n$ of size at most $k^u\cdot l$.
For every $n$ compute (the least) stage $s$ and (the canonical least) finite set $U_n\subseteq 2^n$ with size $k^u\cdot l$ such that
for every $V\subseteq 2^n$, either $Q^n_V[s]=\emptyset$ or $U_n\cap V\neq \emptyset$. Then define $g(n) = U_n$.
Since $Q^n_V $ is $\Pi_1^{0,X}$, the assertion ``$Q^n_V=\emptyset$'' is $\Sigma^{0,X}_1$ and  hence $g\leq_T X$.
By Claim \ref{claim2}, $Q^n_{S\cap 2^n}\ne\emptyset$ for all $n\in\omega$, and so $S\cap 2^n\cap U_n\neq \emptyset$.
Thus, $g(n)$ is a $k^u\cdot l$-enumeration of $S$, contradicting  the assumption that $X$ does not compute a bounded enumeration of $S$.
\qed

\vskip.2in

The above produces a contradiction for Case II that we need.

\subsection{$\Pi_1^0$-class avoidance of bounded enumeration of $\tt^1_k$}
\label{secpi10}

The final part of the proof of our main result Theorem \ref{th101} is to show  the following $\Pi_1^0$-class avoidance property of $\tt_2^1$:

\begin{theorem} \label{th102}
$\tt_2^1$ admits $\Pi_1^0$-class avoidance of bounded enumeration.
\end{theorem}

Theorem \ref{th100} follows from repeated applications of Theorem \ref{th102} in the usual way.
\classavoidance*

\begin{proof}
Let $Q$ be a nonempty $\Pi_1^0$-class of $k$-colorings of $2^{<\omega}$.
For each $C\in Q$, define a $2$-coloring $\h{C}$ such that pairs of nodes with $C$-colors in
$\{1,\dots,k-1\}$ are colored 1 by $\h{C}$, while pairs with $C$-color 0 retain their color.
This yields a $\Pi^0_1$-class $\h{Q}$ of $2$-colorings.
By Theorem \ref{th102}, there is a $C_0\in Q$ such that $\h{C}_0\in\h{Q}$ has a solution $\h{T}$,
i.e.~homogeneous for $\h{C}_0$ that does not compute a bounded enumeration of $S$.
If $\h{T}\subseteq \h{C}_0^{-1}(0)$, we are done. If $\h{T}\subseteq \h{C}_0^{-1}(1)$,
consider $\tilde{Q} = \{C\restriction \h{T}: C\in Q\wedge \h{T}\subseteq C^{-1}(\{1,\dots,k-1\})\}$.
This is a $\Pi_1^{0, \h{T}}$-class of $(k-1)$-colorings of $\h{T}$, and $\tilde{Q}\ne\emptyset$
since $C_0\restriction \h{T}\in \t{Q}$. The conclusion follows from inductive hypothesis.
\end{proof}

We now proceed with the proof of Theorem \ref{th102}.  First, we introduce the combinatorial notions which will be  used subsequently,
and then define the class of conditions and associated  notion of forcing  needed in subsection \ref{sectforcing}
The notion  of extension of a condition is introduced in subsection \ref{sectextension}
 and applied to show that all the requirements are satisfied (Lemmas \ref{lemcore}).

Fix $D\subseteq\omega$ and $S\subseteq 2^{<\omega}$ such that there is no $D$-computable bounded enumeration of $S$.
Let $Q$ be a nonempty $\Pi_1^{0,D}$-class of $\tt_2^1$ instances.
We prove that there exist an instance $C\in Q$ and a $\tt_2^1$-solution $G$ of $C$ such that $D\oplus G$ does not compute a bounded enumeration of $S$.
For simplicity, assume $D = \emptyset$ and any reference to it  is henceforth suppressed.
In the proof of Theorem \ref{th102} we will assume that for every Turing functional $\Psi_e$, there exists an $l$
such that for every $X$, $\Psi_e^X$ computes an $l$-enumeration of $2^{<\omega}$, and for all $X$ and $n$,
$\Psi_e^X(n)\downarrow$ implies $\Psi_e^X(n)\subseteq 2^n$.

\subsubsection{Preconditions} \label{sectforcing}
In order to apply the idea in Liu \cite{Liu2015Cone}, we consider ``$d$-dimensional coloring vectors'' $\v{C}=(C_0,\dots,C_{d-1})$
where $d$ is a positive integer and each $C_n$ ($n<d$) is an instance of $\tt^1_2$.

Given a $d$-dimensional coloring vector $\v{C}$, a subset $I\subseteq d=\{0,\dots, d-1\}$, a color $i\in 2$,
we say that a tree $F$ is \emph{colored $i$ on $I$} by $\v{C}$, if for every $n\in I$, $F$ is colored $i$ by the coloring $C_n$.
Given a $\Pi_1^0$-class $\v Q$ of $d$-dimensional coloring vectors, we say that $F$ is \emph{colored $i$ on $I$ by $ Q$},
if for every $\v{C}\in \mbq$, $F$ is colored $i$ on $I$ by $\v{C}$.


A \emph{precondition} $p$ is a tuple $(F,B,i,I,\mbq)$ where $\mbq$ is a nonempty $\Pi_1^0$-class of $d$-dimensional coloring vectors,
$i\in 2$, $I\subseteq d$, $F$ is a finite perfect tree which is colored $i$ on $I$ by $\mbq$, and $B$ is a 1-1 cover of $\leaf(F)$,
meaning $B$ is a cover of $\leaf(F)$ and there is a bijection between $B$ and $\leaf(F)$.
Intuitively,  one may consider a precondition to be  a collection of   potential Mathias-type conditions $(F,X)$,where $X$ is a forest in $[B]^{\preceq}=\bigcup \{[\rho]^{\preceq}: \rho\in B\}$
  and contains  a ``subforest'' colored $i$ in some $\vec C\in \vec Q$,  for sone $i\in I$.

We say that one precondition $p'= (F',B',i',I',\mbq')$ \emph{extends} another $p=(F,B,i,I,\mbq) $ (written as $p'\leq p$)
if $B'$ covers $B$, $F'\succeq F, F'\setminus F\subseteq [B]^\preceq$, $i'=i,I'=I$ and $\mbq'\subseteq \mbq$.

The requirements on avoidance of bounded enumeration are as follows:
\begin{itemize}
\item $\mcal{R}_e$: Either for some $n$, $\Psi_e^G(n)\downarrow$ and $\Psi_e^G(n) \cap S=\emptyset$ or $\Psi_e^G$ is not total.
(Note that the previous convention concerning the type of outputs  $\Psi_e$may compute  continues to  apply.)
\end{itemize}

We say that a precondition $(F,B,i,I,\mbq)$ \emph{satisfies} $\mcal{R}_e$ if there is an $n$ such that $\Psi^F_e(n)\downarrow$ and $\Psi^F_e(n)\cap S=\emptyset$,
or for all $(F',B',i',I',\v Q')\leq (F,B,i,I,\v Q)$, $\Psi^{F'}_e(n)\uparrow$ (this is property is closer to Cohen forcing than Mathias forcing).

There is  another set of requirements to ensure that the generic tree  $G$ is  infinite:
\begin{itemize}
\item $\mcal{P}_e$: There is an $F\preceq G$ such that $F\cong 2^e$.
\end{itemize}

A precondition $(F,B,i,I,\v Q)$ \emph{satisfies} $\mcal{P}_e$ if $F\cong 2^{<n}$ for some $n\geq e$.

Suppose   $p=(F,B,i,I,\v Q)$ is a precondition  such that  every component of $\vec Q$ is from the given $\Pi^0_1$-class and for any  $p'\leq p$,
for any index $e$, there is a $p''\leq p'$ which satisfies $\mcal{R}_e$ and $\mcal{P}_e$. Then an easy argument similar to the proof of Lemma \ref{lem18} will complete  the proof of Theorem \ref{th102}.
Hence  from now on we  assume
\vskip.2in
($\dagger$) For any  $p$ there is a  $p'\leq p$ and a least index $e=e(p)$
such that any extension of $p'$ fails to satisfy either $\mcal{R}_e$ or $\mcal{P}_e$.

\vskip.2in

We will refer to such a $p'$ as  being {\it bad} for $\mathcal{R}_e$ or $\mcal{P}_e$ accordingly.
Clearly if $p$ satisfies or is bad for a requirement, then all extensions of $p$ also satisfy or are bad for the same requirement.
This simple fact will be used several times when performing Type-I extensions.  Call this Fact (*).

We now introduce the notion of a condition, which is essentially a family of preconditions labelled by the colors on the associated  perfect tree $F$.
Prior to this, there is the notion of   a family of ``density predicting'' functions $\mcal{G}_{\v Q}$.
Let $T\subseteq 2^{<\omega}$ be a tree.  Recall that a set $X\subseteq T$ is {\em dense in} $T$
if for any $\sigma\in T$ there is a $\tau\in X$ extending $\sigma$.
For a node $\rho\in T$, we say that $X$ is {\em dense over $\rho$} if $X\cap [\rho]^{\preceq}$ is dense in $T\cap [\rho]^{\preceq}$.
$X$ is \emph{somewhere dense over} $\rho$ and $X$ is \emph{nowhere dense over} $\rho$ are defined similarly.

Let $\mbq$ be a collection of $d$-dimensional coloring vectors.
Define the family $\mathcal{G}_{\mbq}$ of functions $g: 2^{<\omega}\rightarrow 2^d$ as follows:
\begin{align*}
\mathcal{G}_{\mbq} = \big\{ g: \ \ &\text{ There exists a } (C_0,\dots,C_{d-1})\in\mbq \text{ such that for every }\sigma, \\ \nonumber
&\bigcap\limits_{n<d} C_n^{-1}\big(g(\sigma)(n)\big) \text{ is \somewheredense\ over }\sigma. \big\}
\end{align*}
Informally, every $g\in \mcal{G}_{\v Q}$ has a witness vector $\v{C}$ such that for each $\sigma$ on the full binary  tree,
$g(\sigma)$ selects a ``color vector'' $\zeta$ such that the nodes colored  $\zeta$ is dense in a cone above $\sigma$,
thus producing a perfect tree which is simultaneously homogenous for $C_0,\dots,C_{d-1}$. Note that those $\zeta$'s  need not be consistent.

\begin{lemma} \label{lem33}
Let $\mbq$ be a collection of $d$-dimensional coloring vectors. Then
\begin{enumerate} [{\rm (i)}]
\item $\mathcal{G}_{\mbq}$ has the following closure property: Given $g_1\in \mcal{G}_{\mbq}$,
 antichains $E_1, E_2\subseteq 2^{<\omega}$ such that $E_1$ covers $E_2$,
and given a function $h$ defined on $E_2$ such that every $\sigma\in E_2$ has an extension $\tau\in E_1$ with $h(\sigma)=g_1(\tau)$,
there exists a $g_2\in\mathcal{G}_{\v Q}$ which extends $h$.

\item If $\mbq'\subseteq \mbq$, then $\mathcal{G}_{\mbq'}\subseteq \mathcal{G}_{\mbq}$;

\item $\mbq\ne\emptyset$ implies $\mathcal{G}_{\mbq}\ne\emptyset$;

\item If $\mbq= \mbq_0\times\cdots\times\mbq_{n-1}=\{(\v{C}_0, \dots,\v{C}_{n-1}): \v{C}_m\in\mbq_m \}$,
then $\mathcal{G}_{\mbq}\subseteq \mathcal{G}_{\mbq_0}\times\cdots \times\mathcal{G}_{\mbq_{n-1}}$.
Here we identify $X^{m_0}\times \dots \times X^{m_{n-1}}$ with $X^{m_0+\dots+m_{n-1}}$ where $X$ is the set of $2$-colorings.
\end{enumerate}
\end{lemma}

\begin{proof}
We only show (i) and (iii), as the other two are immediate.

For (i), let
\[
g_2(\sigma)=\left\{
              \begin{array}{ll}
                g_1(\sigma), & \hbox{if $\sigma\not\in E_2$;} \\
                h(\sigma), & \hbox{$\sigma\in E_2$.}
              \end{array}
            \right.
\]
Clearly $g_2$ extends $h$.  To see that $g_2\in \mathcal{G}_{\mbq}$, we use the same witness vector $(C_0,\dots,C_{d-1})$ for $g_1$.
It remains to verify the defining property for $\sigma\in E_2$.  Let $g_2(\sigma)$ be $\zeta\in 2^d$.  Then by definition,
$h(\sigma)=\zeta$.  So for some $\tau\in E_1$ extending $\sigma$, $g_1(\tau)=\zeta$.  Since $g_1\in \mathcal{G}_{\mbq}$,
$\zeta$ is dense on some cone above $\tau$.  Since $\sigma\preceq \tau$, that cone is also above $\sigma$, and we are done.

For (iii), let $(C_0,\dots,C_{d-1})$ be a coloring vector in $\mbq$.  Notice that for any $\sigma\in 2^{<\omega}$, the coloring vector
restricted to $[\sigma]^{\preceq}$ is a $2^d$-coloring.  Now in an  $\omega$-model, $I\Sigma^0_2$ holds and hence
there is a  color vector $\zeta\in 2^d$ such that $\zeta$ is dense somewhere above $\sigma$.  Let $g(\sigma)=\zeta$ and  we are done.
\end{proof}

\begin{definition} \label{defcondition}
A \emph{condition} is a tuple $q=(\{\mcal{F}_{i,I}\}_{i\in 2, I\in \mcal{I}}, \mcal{I}, d, E, \mbq)$ where
\begin{enumerate}
\item [{\rm (i)}] $d>0$ and $\mcal{I}\subseteq \mcal{P}(d)$;
 \item [{\rm (ii)}] $\mbq$ is a nonempty $\Pi_1^0$-class of $d$-dimension coloring vectors;
 \item [{\rm (iii)}]
Each $\mcal{F}_{i,I}$ is a finite set of pairs $(F,B)$ such that $F$ is colored $i$ on $I$ by $\mbq$ and $B$ is a cover of $F$;
\item [{\rm (iv)}] $E$ is a finite antichain on $2^{<\omega}$ such that for all $i\in 2, I\in \mcal{I}$ and $(F,B)\in \mcal{F}_{i,I}$, $B\subseteq E$, and
\item [{\rm (v)}]  (Sufficiency condition) For every $g\in \mcal{G}_{\v Q}$,  there exist an  $i\in 2$, $I\in \mcal{I}$ and an
$(F,B)\in\mcal{F}_{i,I}$, such that for any $\sigma\in B$ and $n\in I$, $g(\sigma)(n)=i$.
\end{enumerate}
\end{definition}

We will elucidate the role played by $\mcal{I}$  in the discussion of  Type II extensions below.
For Type I extensions, one can safely think of $I\in \mcal{I}$ as $I\subseteq d$.
In a condition $q$, for every $i\in 2, I\in \mcal{I}$ and $(F,B)\in\mcal{F}_{i,I}$, $p=(F,B,i,I,\mbq)$ is a precondition.
We  refer to this as ``$p$ {\em occurs in} $q$''.

The intuition  for sufficiency  is that  for any ``density prediction'' $g\in \mathcal G_{\vec Q}$,
there exist  $i,I,F,B$ such that $F$ is colored $i$ on $I$ and $g$ says $F$ can be extended to an infinite tree of color $i$ on $I$ in the open neighborhood $[B]^{\preceq}$.
This is summarized in  the following lemma.

\begin{lemma} \label{lempos}
For any  $q=(\{\mcal{F}_{i,I}\}_{i\in 2, I\in \mcal{I}}, \mcal{I}, d, E,\mbq)$, there exist $i\in 2, I\in \mcal{I}$,  $(F,B)\in \mcal{F}_{i,I}$,
 $\v{C}\in\mbq$ and an  infinite perfect $G\succ F$ such that  $G\setminus F\subseteq [B]^\preceq$
and $G$ is colored $i$ on $I$ by $\v{C}$.
\end{lemma}

\begin{proof}
Let $g\in\mathcal{G}_{\mbq}$ which exists by Lemma \ref{lem33} (iii), and let the coloring vector $\v{C}=(C_0,\dots,C_{d-1})$ witness  $g$ being in $\mathcal{G}_{\mbq}$,
i.e.~for every $\sigma$,
\begin{align}\label{tt22vswkleq0}
\bigcap\limits_{n<d}C_n^{-1}\big(g(\sigma)(n)\big) \text{ is somewhere dense over }\sigma.
\end{align}
By the sufficiency of $q$, there exist $i\in 2, I\subseteq d$ and $(F,B)\in\mcal{F}_{i,I}$ such that
for every $\sigma\in B$ and $n\in I$, we have $g(\sigma)(n) = i$.
By (\ref{tt22vswkleq0}), for every $\sigma\in B$, $\bigcap_{n\in I}C_n^{-1}(i)$ is somewhere dense over $\sigma$, hence the conclusion.
\end{proof}

We now define the notion of extension of a condition.  To  motivate,
fix a $\Pi_1^0$-class $\v Q$ of instances of $\tt_2^1$.
Let  $q^0=(\{\mathcal{F}_{0, \{0\}},\mcal{F}_{1,\{0\}}\}, \{0\}, 1, \{\varepsilon\},\mbq)$,
where $\mcal{F}_{0,\{0\}} =\{(\emptyset,\emptyset)\},\mcal{F}_{1,\{0\}}=\{ (\emptyset,\emptyset)\}$.
Observe  that $q^0$ is sufficient: For each  $g\in \mcal{G}_Q$, since $g(\varepsilon)$ is a string of length $1$, we may assume that $g(\varepsilon)=\langle0\rangle$.
Then take $i=0$ and $I=\{0\}$, $(\emptyset,\varepsilon)\in\mcal{F}_{i,I}$, and $E'=E=\{\varepsilon\}$, we have for any $n\in I$ (which means $n=0$),
$g(\varepsilon)(n)=0=i$. Intuitively, this is almost trivial: Some coloring $C$ will witness $g$'s prediction.
If $g$ says color $0$ is somewhere dense in a cone above $\varepsilon$, then the precondition in $\mcal{F}_{0,I}$ wins, and  otherwise $\mcal{F}_{1,I}$ wins.

By ($\dagger$), the two preconditions occurring in $q^0$ can be extended to preconditions that are bad for either some $\mcal{R}$ or some $\mcal{P}$ (with least index).
Sufficiency of $q^0$  tells us that one of them is bad for $\mcal{R}_e$ for some $e$.
We first argue that these extensions of preconditions will not destroy  sufficiency.
This will be our {\em Type I extension}.   Following that  is an  extension by another    condition $q'$ in which a precondition $p$ is bad for some $\mcal{R}$.
The next step is to increase the dimension of the coloring vectors   to obtain  a condition $q''$ in which all these bad $\mcal{R}$'s are now  satisfied.
This will be our {\em Type II extension}.
 Besides satisfying the requirement, preserving  sufficiency will  also be a  key objective in the construction.

\subsubsection{Type I Extensions} \label{sectextension}
Let $q=(\{\mcal{F}_{i,I}\}_{i\in 2,I\in \mcal{I}},\mcal{I}, d, E, \mbq)$ be a condition.
A Type I extension of $q$ consists of a series of actions on  preconditions $p$ occurring in $q$ so that $p$ is either removed from $q$   in the Type I extension  or is  extended to a $p'$ which is bad for some $\mcal{R}_e$.
In fact any requirement  higher than  $\mcal{R}_e$ in the priority list is satisfied by $p'$.

The action on each precondition is organized in the order of   listing of   pairs
$(i, I)$.
Fix a pair $(i,I)$ with $i\in 2$ and $I\in \mcal{I}$. Let $(F_0,B_0), \dots, (F_r,B_r)\ne\emptyset$ be all the members in $\mcal{F}_{i,I}$.
By Fact (*), we may ignore those which are   bad for some $\mcal{R}_e$.
We claim that  one may  remove the preconditions  which are bad for some $\mcal{P}_n$ while preserving sufficiency.

Suppose that $(F, B)\in \mcal{F}_{i,I}$ is bad for some $\mcal{P}_n$. In other words,
there is no perfect \treeset\ $F'\cong 2^{<n}$, where $n>\text{ the height of}$ $F$,  such that $F'\succeq F, F'\setminus F\subseteq [B]^\preceq$,
and for some coloring vector $\v{C}\in\mbq$, $F'$ is colored $i$ on $I$ by $\v{C}$.
In this case, we remove  $(F,B)$ from $\mcal{F}_{i, I}$. Call the resulting family $\mcal{F}'_{i,I}$.
(Note that $\mcal{F}'_{i,I}$ could be empty  and thus will not be considered in henceforth in the constructi0n).
Let $q'$ be the condition obtained from $q$ by replacing $\mcal{F}_{i,I}$ with $\mcal{F}'_{i,I}$. We show that $q'$ is a condition by showing its sufficiency.

Let $g\in \mcal{G}_{\v Q}$ with witness $\v{C}\in  \v Q$. Since $q$ is sufficient, there exist some $i^*\in 2$ and $I^*\subseteq d$,
$(F^*,B^*)\in\mcal{F}_{i^*,I^*}$ such that $B^*\subseteq E$ and for any $\sigma\in B^*$ and $n\in I^*$, $g(\sigma)(n)=i^*$.
By the definition of $\mcal{G}_{\v Q}$, for every $\sigma\in B^*$, the subtree $\subseteq [\sigma]^{\preceq}$ colored  with $i^*$ on $I^*$ is somewhere dense above $\sigma$.
Clearly either the pair $(i^*, I^*)$ differs from $(i,I)$ or $(F^*, B^*)\neq (F,B)$ by the assumption on $(F,B)$.  Thus $q'$ is sufficient. For simplicity of notations, suppose $q'=q$.

Let $\mathcal R_e, \mathcal P_n$, where $e, n\in\omega$, be listed in order of priority.  Starting with $(F_0, B_0)$, one   extends it in sequence  to satisfy the next unsatisfied requirement $\mcal{R}_e$ or $\mcal{P}_n$, until it is no longer possible to do so. This would be an $\mathcal R$-requirement, say $\mathcal R_{e_0}$, that is not satisfiable.
Let the resulting element be denoted $(F_0', B_0')$. Thus  $F_0'$ is a finite perfect \treeset\ such that $F_0'\succeq F_0, F'_0\setminus F_0\subseteq [B_0]^\preceq$,
and for some coloring vector $\v{C}\in\mbq$, $F'_0$ is colored $i$ on $I$ by $\v{C}$. Furthermore,  $F_0'$ satisfies every $\mcal{R}_e$ and $\mcal{P}_n$ of higher priority than  $\mcal{R}_{e_0}$,
and no extension via a  precondition  satisfies  $\mcal{R}_{e_0}$, and  $B_0'=\leaf(F_0')$.
Let $\mbq_0=\{\v{C}\in \mbq: F'_0$ is colored $i$ on $I$ by $\v{C}\}$. Now do the same for $(F_1,B_1)$ but replace $\mbq$ with  $\mbq_0$.
Performing this series of operations  over $j'\leq r$, one obtains   $(F_j',B_j')$  and $\mbq'=\mbq_{r}$ which is a nonempty subclass of $\mbq$.

Let $E_1$ be an antichain which covers $\bigcup_{j=0}^{r} \leaf(F_j')$, $E_2=\{\tau\in E: \tau$ is not below any node in $\bigcup_{j=0}^{r} \leaf(F'_j)\}$
and $E'=E_1\cup E_2$, $\mcal{F}'_{i,I}=\{(F'_j, B''_j): j<r\}$ where $B''_j=\{\tau\in E':$ for some $\sigma\in \bigcup_{0\leq j<r} B_j'$,
$\tau$ is the leftmost extension of $\sigma$ in $E'\}$.

For pairs $(i',I')\ne (i,I)$, replace $(F,B)\in \mcal{F}_{i', I'}$ with $\{(F,\h{B}): \h{B}\subseteq E'$ and $\h{B}$ is a 1-1 cover of $\leaf(F)\}$,
Let $\mcal{F}'_{i',I'}$ denote the resulting family of pairs.  Note that $\mcal{F}_{i',I'}$ now may contain  more members than before,
because of the different choices for $\hat B$, for  each given  $F$.

We claim that $q'=(\{\mcal{F}'_{i,I}\}_{i\in 2,I\in \mcal{I}}, \mcal{I}, d, E', \mbq')$ is  a condition.
The only nontrivial part is to verify the sufficiency of $q'$..

Let $g\in \mcal{G}_{Q'}$ and $A=\{\zeta\in 2^d: \forall n\in I, \zeta(n)=i\}$.
Consider the following function $h: E\rightarrow 2^d$,
\begin{equation}
h(\sigma) =\left\{\begin{aligned}
&g(\tau) \text{\ for the leftmost such\ }\tau, &\text{\ \ if there is some\ }\tau\in E', \sigma\preceq \tau \text{\ and\ } g(\tau)\not\in  A; \\ \nonumber
&g(\tau) \text{\ for the leftmost $\tau\succeq \sigma$ with $\tau\in E'$,} &\text{\  \ otherwise.}
\end{aligned}\right.
\end{equation}
Note that for every $\sigma\in E$ there exists a $\tau\in E'$ with $\tau\succeq \sigma$ such that $g(\tau) = h(\sigma)$.
By the closure property (Lemma \ref{lem33} (i)) of $\mcal{G}_{\mbq'}$, $h$ has an extension $\h{g}$ in $\mcal{G}_{\mbq'}$.
Note also that $\mcal{G}_{\mbq'}\subseteq \mcal{G}_{\mbq}$  by Lemma \ref{lem33} (ii) since $\mbq'\subseteq \mbq$ .
Hence $\h{g}\in \mcal{G}_{\mbq}$.

Since $q$ is sufficient, there exist  $i^*\in 2$ and $I^*\in \mcal{I}$,
$(F^*,B^*)\in\mcal{F}_{i^*,I^*}$ such that $B^*\subseteq E$ and for any $\sigma\in B^*$ and $n\in I^*$, $\h{g}(\sigma)(n)=i^*$.

If the pair $(i^*, I^*)$ is different from $(i, I)$, then by lifting $\sigma\in B^*$ to some $\tau_{\sigma}\in E'$ with $g(\tau_{\sigma})=h(\sigma)=\h{g}(\sigma)$,
we obtain  $\t{B}=\{\tau_{\sigma}\in E': \sigma\in B^*\}$ which covers $B^*$ and for any $\tau\in \t{B}$ and $n\in I^*$,
we have $g(\tau_{\sigma})(n)=\h{g}(\sigma)(n)=i^*$.  In other words, $i^*, I^*$ $F^*$ and $\t{B}$ witness that $q'$ is sufficient
(this is the reason that we take $(F^*, \t{B})\in \mcal{F}_{i^*, I^*}$ for all 1-1 covers $\t{B}$ of $B$).

If $i^*=i$ and $I^*=I$, then $(F^*, B^*)$ is some $(F,B)\in \mcal{F}_{i,I}$.
Consider $\sigma\in B$.  By the definition of $h$, for any $\tau\succeq \sigma$ with $\tau\in E'$,
we have $g(\tau)\in A$.  Thus for any $\tau\in B'$, $g(\tau)\in A$. Then  $(F', B')$ to witness the sufficiency of $q'$.
This completes the verification of the sufficiency of $q'$.

The above steps  take care of one pair $(i, I)$. By Fact (*), we can proceed  to the next pair $(i', I')$ without worrying further about  the pair $(i,I)$.
Upon going through each pair $(i, I)$, one obtains  a condition $q^*$ such that every precondition $p$ occurring in $q^*$ is bad for some (first in the priority list) requirement $\mcal{R}_e$
with witness $\Psi_e$ and $n$.

\subsubsection{Type II Extensions}
We have noted that each precondition $p=(F,B,i,I,\v Q)$ resembles   a Mathias condition $(F,X)$ where  $X$ is indirectly controlled by $\v Q$,
in the sense that $X\subseteq [B]^{\preceq}$ is a finite union of trees colored $i$ on $I$ for all $\vec C$ in $\v Q$.
When $p$ is bad for some $\mcal{R}_e$ with witness $\Psi_e$ and $n$,
it may be the case that for any $F'\succeq F$, there is an  $F''\succeq F'$ such that $(F''\setminus F)\subseteq X$ and $\Psi_e^{F''}(n)\downarrow\cap\ S\ne \emptyset$.
A way to  overcome this difficulty is to ``thin'' $X$ to a  $Y$  so that  all finite trees $F''$ with $(F''\setminus F)\subseteq Y$ satisfy  $\Psi_e^{F''}(n)\uparrow$.
This thinning  of $X$  can be achieved by imposing additional  restrictions on the choice of  coloring vectors.  For example, suppose $d=3$, $i=0$, $I=\{0,1\}$,
and the Mattias condition $(F,X)$ is such that  for all $(C_0,C_1,C_2)\in \v Q$, both $F$ and $X$ are colored $0$ by  $C_0$ and $C_1$,
while no requiured color is specified under  $C_2$.  To thin $X$,  one may choose  subclasses $\v Q', \v Q''$ of $\v Q$, and
form coloring vectors $(C_0',C_1',C_2'; C_0'',C_1'',C''_2)$ where $(C_0',C_1',C_2')\in \v Q'$ and $(C_0'',C_1'',C''_2)\in \v Q''$;
and we  choose  $Y$ to be colored $0$ under $C_0',C_1', C_0'',C_1'',C''_2$ and any color under  $ C_2'$. In this way $Y$ is thinned from $X$ due to the additional  requirement  that $C''_2(Y)=\{0\}$.
If $\v Q=\v Q'$, then indeed $Y\subseteq X$.
This example gives the motivation for introducing the notion of a {\it   Type II extension}.
Let $q=(\{\mcal{F}_{i,I}\}_{i\in 2, I\in \mcal{I}}, \mcal{I}, d, E, \mbq)$ be a condition. Let $m>0$ and $\mcal{K}\subseteq \mcal{P}(m)$.
We set some parameters as follows:

\begin{enumerate}
\item [(1)] Consider the cartesian product $\v Q^m$ whose members are  identified with $m\cdot d$-dimensional coloring vectors.
Thus, for a subset $K$ of $m=\{0,\dots,m-1\}$, it makes sense to say $F$ being colored $i$ on $I\times K$ by $\v Q^m$.
The colletion $\v Q'$ is a subclass of $\v Q^m$ with $(\v Q')^{[0]}=\v Q$, namely the first column of $\v Q'$ is $\v Q$.
\item [(2)] Form the cartesian product $\mcal{I}\times \mcal{K}$, and  identify its members  $(I,K)$ with $I\times K$.
Also assume that $0\in K$ for all $K\in \mcal{K}$.
\item [(3)] For each $(F,B)\in \mcal{F}_{i,I}$, each $K\in \mcal{K}$, form a new precondition $p'=(F, B, i, I\times K, \v Q')$.
\end{enumerate}
The intuition is that $p'$ is identified with some $p\in \mcal{F}_{i,I}$ with its part above $B$  specified  by the coloring vectors in $Q'$.
We say  in this case that $p'$ is a{\em Type II} extension  of $p$.

\begin{lemma}\label{lemext}
If $p'$ is a Type II extension  of $p$ and $p$ satisfies requirement $\mcal{R}_e$, then so does $p'$.
\end{lemma}

\begin{proof}
Suppose that $p'=(F,B,i, I\times K, v Q')$ is a Type II extension of  $p=(F,B, i, I, \v Q)$.
It suffices to show that any extension $\t{p}$ of $p'$  with respect to $\v Q'$ is also an extension of $p$   with respect to $\v Q$.
Let $\t{p}=(\t{F}, \t{B}, i, I\times K, \t{Q})$.  Then there exists
\[
(\v{C}_0,\dots,\v{C}_{m-1})=(C_{0,0},\cdots,C_{0,m-1},C_{1,0},\cdots, C_{1,m-1},\cdots,C_{d-1,0},\cdots, C_{d-1,m-1})\in \v Q'
\]
such that $\t{F}\setminus F'\subseteq \bigcap\limits_{n\in I,m\in K} C_{n,m}^{-1}(i)$.  Since $\v Q$ is the first copy of $\v Q^m$ and $0\in K$, we have in particular
\[
\t{F}\setminus F\subseteq \bigcap\limits_{n\in I} C_n^{-1}(i),
\]
and the conclusion  follows.
\end{proof}

We now return  to the end of the last subsection, where Type I extensions have completed, arriving at  a condition $q=(\{\mcal{F}_{i,I}\}_{i\in 2,I\in\mcal{I}},\mcal{I},d, E, \mbq)$,
such that for each precondition $p$ occurring in $q$, there are natural numbers $e(p), n(p)$ where  $p$ is bad for $R_{e(p)}$  witnessed by $n(p)$.
For simplicity, we implicitly assume that the $\v Q^*$ below contains $\v Q$ as its first column and any $K\in \mcal{K}$ has $0$ as its member.

\begin{lemma}\label{lemcore}
Given $q$ as above, there exist $m$, $\mcal{K}\subseteq \mathcal{P}(m)$ and $\v Q^*\subseteq\v Q^m$ such that
$$
q^*=(\{\mcal{F}_{i,I\times K}^*\}_{i\in 2,I\times K\in\mcal{I\times K}},\mcal{I\times K},d\cdot m,E,\mbq^*)
$$
is a condition and every $p'$ occurring in $q^*$ is a Type II extension    of some $p$ occurring in $q$
and $p'$ satisfies $\mcal{R}_{e(p)}$.
\end{lemma}

\begin{proof}
The proof runs in parallel with that of Case (II) in the previous subsection.
Given $n$ and $V\subseteq 2^{<n}$, consider the following $\Pi_1^0$-class:
\begin{align}
\mbq^n_V =\big\{(C_0,\dots,C_{d-1})\in \mbq: &\text{ For every } p=(F,B,i,I,\v Q) \text{ occurring in } q, \text{ we have:}\\ \nonumber
&\ \text{ for every finite perfect \treeset\ }\h{F}\succeq F \text{ such that }\h{F}\setminus F\subseteq[B]^\preceq,\\ \nonumber
&\ \text{ if }\h{F}\subseteq \bigcap\limits_{j\in I} C_{j}^{-1}(i), \text{ then } \Psi^{\h{F}}_{e(p)}(n(p))\downarrow\rightarrow\Psi^{\h{F}}_{e(p)}(n(p))\cap V\ne\emptyset
\big\}.
\end{align}

We first make a claim:

\begin{claim}\label{claim1}
 For any $n\in\omega$, $\mbq^n_{S\cap 2^{<n}}\ne \emptyset$.
\end{claim}
\noindent  {\it Proof of Claim.}
Suppose otherwise and let $n$ be a counterexample. Fix a $(\t{C}_0,\dots,\t{C}_{d-1})\in\mbq$.
Since $(\t{C}_0,\dots,\t{C}_{d-1})\notin\mbq_{S\cap 2^{<n}}$, there exist an $i\in 2,I\in\mcal{I},(F,B)\in\mcal{F}_{i,I}$,
and a finite perfect \treeset\ $\h{F}\succeq F$ such that $\h{F}\setminus F\subseteq [B]^\preceq$ and $\h{F}\subseteq \bigcap_{j\in I}\t{C}_{j}^{-1}(i)$,
with the property that $\Psi^{\h{F}}_{e(p)}(n)\downarrow\cap S\cap 2^{<n}=\emptyset$, where $p$ is the precondition $(F,B,i,I,\v Q)$ which occurs in $q$.
This, however , contradicts the fact that $p$ has no extension that satisfies $\mcal{R}_{e(p)}$.
\qed
\vskip.15in

Suppose $|E| = u$ and let $\t{l}=\max\{l:$ for some $p$ occurring in $q$, $\Psi_{e(p)}$ is an $l$-enumeration$\}$. We divide the proof of the lemma into two cases.
\vskip.15in

\noindent\textbf{Case 1}.
For some $n$, $\{V\subseteq 2^{{<n}}:\mbq^n_V\ne\emptyset\}$ is $(2^{ud},\t{l})$-\disperse.
\vskip.15in

Suppose $\{V\subseteq 2^{<n}:\mbq_V\ne\emptyset\}= \{V_0,\cdots,V_{m-1}\}$.
Define $\mcal{K}=\big\{K\subseteq m: \{V_{k}\}_{k\in K}\text{ is }\t{l}\text{-\disperse}\big\}$;
$\mbq^* = \mbq_{V_0}\times \cdots\times \mbq_{V_{m-1}}$;
and $\mcal{F}_{i,I\times K}^* =\mcal{F}_{i,I}$ for each $I\in\mcal{I}$ and $K\in\mcal{K}$.

\vskip.2in
We  verify that $q^*=(\{\mcal{F}_{i,I\times K}^*\}_{i\in 2,I\times K\in\mcal{I}\times\mcal{K}},\mcal{I}\times\mcal{K},d\cdot m,E,\mbq^*)$ is \sufficient,
i.e.~for any $g\in\mathcal{G}_{\mbq^*}$, there exist  $\t{i}\in 2$,  $\t{I}\times \t{K}\in\mcal{I}\times \mcal{K}$ and $(\t{F},\t{B})\in\mcal{F}^*_{\t{i},\t{I}\times \t{K}}$
such that for any $\sigma\in \t{B}$ and $(j,k)\in \t{I}\times \t{K}$, $g(\sigma)(j,k) = \t{i}$.

Let $g$ be a function in $\mathcal{G}_{\mbq^*}$.
Since $\{0,1\}^{d\cdot m} = (\{0,1\}^d)^m$, for each $\sigma\in 2^{<\omega}$,
the string $g(\sigma)\in 2^{dm}$ may be viewed as $(\zeta_0,\dots,\zeta_{m-1})$ where each $\zeta_k\in 2^d$.
For each $k<m$, let $g_k(\sigma)=\zeta_k$ be the function $2^{<\omega}\to 2^d$ induced by $g$.

Note that $\mbq^* = \mbq_{V_0}\times \cdots\times \mbq_{V_{m-1}}$,
and so $\mathcal{G}_{\mbq^*}\subseteq \mathcal{G}_{\mbq_{V_0}}\times\cdots \mathcal{G}_{\mbq_{V_{m-1}}}$ by Lemma \ref{lem33} (iii).
Therefore Lemma \ref{lem33} (ii) implies that for every $k<m$, $g_{k}\in\mathcal{G}_{\mbq_{V_{k}}}\subseteq \mathcal{G}_{\mbq}$.

Now regard each $g_{k},k< m$, as a function $E\rightarrow 2^d$.
Since $|E|=u$, there are $2^{ud}$-many functions in $E \rightarrow 2^d$.
Consider the following $2^{ud}$-\partition\ of $m$ indexed by functions $h$ in $E \rightarrow 2^d$:
$$
W_{h} =\big\{k< m:\ \ \text{for any\ } \sigma\in E,\ g_{k}(\sigma)= h(\sigma)\big\}.
$$
Since $\{V_{k}\}_{k< m}$ is $(2^{ud},\t{l})$-\disperse, there exists a $W_{h^*}$ such that $\{V_{k}\}_{k\in W_{h^*}}\ne\emptyset$ is $\t{l}$-\disperse.
By the definition of $\mcal{K}$, $W_{h^*}\in \mcal{K}$.
By the definition of $W_{h^*}$, we see that $h^*=g_k\uhr E$ where $g_k\in \mathcal{G}_{\mbq}$.
By the sufficiency of $q$, there exist $\t{i}\in 2,\t{I}\in\mcal{I}$ and $(\t{F},\t{B})\in\mcal{F}_{\t{i},\t{I}}$
such that for any $j\in \t{I}$ and $\sigma\in \t{B}\subseteq E$, $h^*(\sigma)(j) = \t{i}$.
Thus for all  $j\in \t{I}$, $k\in W_{h^*}$ and $\sigma in \t{B}$, we have
$g(\sigma)(j,k) = g_{k}(j)=h^(j)=\t{i}$ (because $g_k\uhr E=h^*$ for all $k\in W_{h^*}$).
Thus $q^*$ is sufficient.  We  call this $q^*$ a {\em Type II extension of $q$}.

\vskip.15in

Next we show that every precondition $p'$ occurring in $q^*$ satisfies $\Psi_{e(p)}$, where $p'$ is a Type II extension  of a precondition $p$.
More precisely, for all $i\in 2,I\times K\in\mcal{I}\times \mcal{K}$ and $(F,B)\in\mcal{F}_{i,I\times K}^*$,
$(F,B,i,I\times K,\mbq^*)$ satisfies $\mcal{R}_{e(p)}$, where $p=(F,B,i,I,Q)$ occurs in $q$.

Fix $i\in 2,I\in\mcal{I}, K\in\mcal{K}, (F,B)\in\mcal{F}_{i,I\times K}^*=\mcal{F}_{i,I}$ and
$\v{C}=(\v{C}_0,\cdots, \v{C}_{m-1})\in\mbq^*$, where $\v{C}_{k} = (C_{k,0},\cdots,C_{k,d-1})\in\mbq_{V_{k}}$ for each $k<m$.
It suffices to show that for every perfect \treeset\ $F'\succeq F$ such that $F'\setminus F\subseteq [B]^\preceq$ and
$B\subseteq \bigcap_{j\in I,k\in K}C^{-1}_{j,k}(i)$, we have $\Psi_{e(p)}^{F'}(n)\uparrow$
(recall that $n$ is the number under the  assumption of Case 1).
By the definition of $\mbq^n_{V_{k}}$ and the fact that $F'\subseteq \bigcap_{j\in I}C_{k,j}^{-1}(i)$ for all $k\in K$,
we have: $\Psi_{e(p)}^{F'}(n)\downarrow\rightarrow\Psi_{e(p)}^{F'}(n)\cap V_{k}\ne\emptyset$ for all $k\in K$.
On the other hand, $\{V_{k}\}_{k\in K}$ is $\t{l}$-\disperse\ and $\Psi_{e(p)}^{F'}(n)\downarrow\rightarrow |\Psi_{e(p)}^{F'}(n)|\leq \t{l}$
by the definition of $\t{l}$. Thus, by Lemma \ref{tt22vswkllem17} (ii), we conclude that $\Psi_{e(p)}^{F'}(n)\uparrow$.
This takes care  of Case 1.

\vskip.15in

\noindent\textbf{Case 2}. The negation of Case 1 holds.

Then for any $n$, $\big\{V\subseteq 2^n: \mbq_{V}\ne\emptyset\big\}$ is not $(2^{ud},\t{l})$-\disperse.
In this case, as argued in Case 2 of  Case (II) in the previous subsection, one can compute a $2^{ud}\cdot\t{l}$-enumeration of $S$,
yielding a contradiction. The proof is omitted.
\end{proof}

Finally, with all the ingredients at hand, we  now  complete the proof of Theorem \ref{th102}.

\begin{proof} [Proof of Theorem \ref{th102}]
Fix a nonempty $\Pi_1^0$-class $\v Q$ of instances of $\tt_2^1$.
We begin  with the condition $q^0=(\{\mathcal{F}_{0, \{0\}},\mcal{F}_{1,\{0\}}\}, $\{0\}$, 1, \{\varepsilon\},\mbq)$,
where $\mcal{F}_{0,\{0\}} =\{(\emptyset,\emptyset)\},\mcal{F}_{1,\{0\}}=\{ (\emptyset,\emptyset)\}$
(recall that we had earlier  verified that $q_0$ is indeed a condition).
Suppose we are given  a condition $q_s = (\{\mcal{F}_{i,I}\}_{i\in2,I\in\mcal{I}},\mcal{I},d, E, \mbq_s)$, where $s$ is even.
First extend $q_s$ using Type I extension to obtain $q_{s+1}$ such that every $p$ occurring in $q_{s+1}$ is bad for some $\mcal{R}_{e(p)}$
(we are assuming here that we are not in the easy situation  where certain $p$ can be extended all the way to satisfy all requirements).
Then apply Type II extension to obtain  $q_{s+2}$ as in Case 1 of the proof of Lemma \ref{lemcore}.
Iterating the procedure yields an infinite sequence of conditions $q_0\geq q_1\geq q_2\geq \cdots$.

Now consider the collection of preconditions
\[
\mcal{T}=\{p: \mbox{ for some $s$, $p$ occurs in the condition $q_s$}\}.
\]
The relation $p'\leq p$ as preconditions gives rise to a tree structure on $\mcal{T}$.
Since $\mcal{T}$ is an infinite and finitely branching tree, it admits an infinite path
\[
\rho:= \{p_0\geq p_1\geq \dots\},
\]
where $p_s=(F_s,B_s,i,I_s,\v Q_s)$ (note that the color $i$ is fixed in the extensions of preconditions).
Let $G= \bigcup_s F_s$.  Since every requirement $\mcal{P}_e$ is eventually satisfied by some $p_s$,
$G$ is an infinite perfect tree.  Similarly, since every requirement $\mcal{R}_e$ is eventually satisfied by some $p_s$,
$G$ does not compute a bounded enumeration of $S$.

Finally, for each $s$, let $\v Q'_s=\{C\in \v Q: C$ is a component in some coloring vector $\v{C}\in \v Q_s\}$.
Since $\v Q'_s$ can be obtained from finitely many parameters (such as $m$, $K$) and finitely many $\Pi^0_1$-classes (such as $Q_{V_k}$),
each $\v Q'_s$ is a nonempty $\Pi^0_1$ class.  Furthermore $\v Q'_{s+1}\subseteq \v Q'_s$ for all $s$.
By the finite intersection property, there is a coloring $C\in \bigcap_{s} \v Q'_s$.  Clearly $C\in \v Q$ and $G$ is colored $i$ by $C$.
This completes the proof of the theorem.
\end{proof}

\section{Conclusion}

We end this paper with some remarks and a list of questions.
Clearly every Ramsey type combinatorial principle defined on $\mathbb{N}$ (such as the Chain/antichain principle $\msf{CAC}$ and the Thin Set principle $ \msf{TS}$)
has a ``tree analog', and so every question about $\mathbb{N}$
concerning these principles  can be asked
about $2^{<\omega}$.
Here we list some  of them.

  Monin and Petey \cite{monin2019srt22} have recently produced an $\omega$-model of $\mathsf{SRT}^2_2$ in which $\mathsf{COH}$ fails.
  This problem has a natural analog for trees (Dzhafarov and Patey \cite{DzhafarovColoring}):
\begin{question}\label{tt22vswklques0}
Does $\msf{STT}_2^2$ imply $\msf{CTT}_2^2$ in an $\omega$-model?

\end{question}
In fact, one does  not even have an answer to  the following question:
\begin{question}\label{tt22vswklques1}
Does $\msf{SRT}_2^2$ imply $\msf{CTT}_2^2$ in an $\omega$-model?

\end{question}


Since $\msf{CTT}_2^2$ can be decomposed  into
$\msf{wCTT}_2^2$ and $k$-$\mathsf{TSP}$, , and
$\msf{wCTT}_2^2$ is  a natural generalization of $\msf{COH}$ to trees,
to resolve the above question\label{tt22vswklques1},
one can begin with investigating a question about solutions of a  $k$-\treesplit:

\begin{question}\label{tt22vswklques2}
Does there exist a $k$-\treesplit\ $f$ such that
 every  instance of $\msf{RT}_2^1$   has a solution
that does not compute
an infinite perfect \treeset\ homogeneous
for $f$ (hence a solution of $f$)?

\end{question}
One conjectures that the answer to the above question is in the affirmative, since  coding an object typically proceeds in a fairly straightforward manner, while there is no obvious way to code
a solution of a $k$-\treesplit\ into that of an $\msf{RT}_2^1$-instance.
Indeed  it seems  that coding a $k$-\treesplit\ through  subsets is impossible,
i.e.~the following question should have a positive answer:
\begin{question}\label{tt22vswklques3}
Does there exist a $k$-\treesplit\ $f$ and an infinite set $X$
such that no  $G\subset X$
 computes a solution of $f$?

\end{question}

Subset coding is more powerful than coding via solution of an    $\msf{RT}_2^1$-instance. Hence
 an affirmative answer to Question \ref{tt22vswklques3} would imply
an affirmative answer to Question \ref{tt22vswklques2}. We can also consider
a weaker coding scheme, namely coding via a fast growing function:
\begin{question}\label{tt22vswklques4}
Does there exist a $k$-\treesplit\ $f$ such that
for every $h\in\omega^\omega$, there exists an
$\h{h}\in\omega^\omega$ with $\h{h}(n)\geq h(n)$ for all $n$,
and  $\h{h}$ does not compute a solution of $f$?

\end{question}
Clearly an affirmative answer to Question \ref{tt22vswklques3} would imply
an affirmative answer to Question \ref{tt22vswklques4}. We conjecture that both questions have a positive
answer.

\bibliographystyle{plain}
\bibliography{bibliographylogic}
\end{document}